\documentclass[a4paper,twoside,11pt,reqno]{amsart}

\usepackage{amsmath,amsfonts,amsbsy,amsgen,amscd,mathrsfs,amssymb,amsthm,bm,enumerate,mathtools,comment, makecell}

\usepackage[
style=alphabetic,
maxnames=99,
maxalphanames=4,
doi=false,
isbn=false,
url=false
]{biblatex}
\usepackage[dvipsnames]{xcolor}
\usepackage[colorlinks=true,citecolor=blue,linkcolor=blue]{hyperref}
\usepackage{fullpage}
\usepackage{booktabs}

\hypersetup{
     colorlinks   = true,
     citecolor    = black
}
\hypersetup{linkcolor=black}
\hypersetup{urlcolor=black}

\newtheorem{theorem}{Theorem}
\newtheorem*{theorem*}{Theorem}
\newtheorem{lemma}[theorem]{Lemma}
\newtheorem{corollary}[theorem]{Corollary}
\newtheorem*{prop*}{Proposition}
\newtheorem{prop}{Proposition}

\newtheorem*{fact*}{Fact}

\theoremstyle{definition}

\newtheorem{question}{Question}

\newtheorem*{conj*}{Conjecture}
\newtheorem*{ex*}{Example}

\theoremstyle{remark}
\newtheorem{rem}{Remark}

\newcommand*{\claimproofname}{Proof of claim}

\DeclareMathOperator{\poly}{poly}

\DeclareMathOperator*{\E}{\mathbf{E}}

\newcommand{\go}{K}

\newcommand{\al}{\alpha}
\newcommand{\eps}{\varepsilon}
\newcommand{\om}{\omega}
\newcommand{\Om}{\Omega}
\newcommand{\supp}{\textnormal{supp}}

\newcommand{\mc}[1]{\mathcal{#1}}
\newcommand{\bs}[1]{\boldsymbol{#1}}

\renewcommand{\Pr}{\textbf{Pr}}

\newcommand{\Lesssim}[1]{\lesssim_{{\textstyle\mathstrut}{#1}}}

\DeclareMathSymbol{\shortminus}{\mathbin}{AMSa}{"39}

\theoremstyle{definition}

\numberwithin{equation}{section} 
\numberwithin{figure}{section}
\numberwithin{table}{section}

\newcommand{\bE}{\mathbf{E}}
\newcommand{\R}{\mathbf{R}}
\newcommand{\C}{\mathbf{C}}

\newcommand{\T}{\mathbf{T}}
\newcommand{\D}{\mathbf{D}}
\newcommand{\N}{\mathbf{N}}
\newcommand{\Z}{\mathbf{Z}}

\newcommand{\iu}{\mathbf{i}}

\newcommand{\x}{\bs{x}}
\newcommand{\y}{\bs{y}}
\newcommand{\z}{\bs{z}}
\newcommand{\W}{w}
\newcommand{\Y}{\mathbf{Y}}
\renewcommand{\a}{\bs{a}}
\renewcommand{\v}{\bs{v}}

\renewcommand{\b}{\bs{b}}

\DeclareMathOperator{\sgn}{\operatorname{sgn}}

\begin{document}
	
	\title[]{Dimension-free Bernstein's discretization Inequalities}
    \title{Dimension-free discretizations of the uniform norm\\by small product sets}
	
\author[L. Becker]{Lars Becker}
\address{(L.B.) Mathematical Institute, 
	University of Bonn,
	Endenicher Allee 60, 53115, Bonn,
	Germany}
\email{becker@math.uni-bonn.de}

\author[O. Klein]{Ohad Klein}
\address{(O.K.) School of Computer Science and Engineering, Hebrew University, Jerusalem}
\email{ohadkel@gmail.com}

\author[J. Slote]{Joseph Slote}
\address{(J.S.) Department of Computing and Mathematical Sciences, California Institute of Technology, Pasadena, CA}
\email{jslote@caltech.edu}

\author[A. Volberg]{Alexander Volberg}
\address{(A.V.) Department of Mathematics, MSU, 
East Lansing, MI 48823, USA, and Hausdorff Center of Mathematics, Bonn}
\email{volberg@math.msu.edu}

\author[H. Zhang]{Haonan Zhang}
\address{(H.Z.) Department of Mathematics, University of South Carolina, Columbia, SC 29208, USA}
\email{haonanzhangmath@gmail.com}
\begin{abstract}
    Let $f$ be an analytic polynomial of degree at most $K-1$.
    A classical inequality of Bernstein compares the supremum norm of $f$ over the unit circle to its supremum norm over the sampling set of the $K$-th roots of unity.
    Many extensions of this inequality exist, often understood under the umbrella of Marcinkiewicz--Zygmund-type inequalities for $L^p,1\le p\leq \infty$ norms.
    We study dimension-free extensions of these discretization inequalities in the high-dimension regime, where existing results construct sampling sets with cardinality growing with the total degree of the polynomial. 
    
    In this work we show that dimension-free discretizations are possible with sampling sets whose cardinality is independent of $\deg(f)$ and is instead governed by the maximum \emph{individual} degree of $f$; \emph{i.e.}, the largest degree of $f$ when viewed as a univariate polynomial in any coordinate.
    For example, we find that for $n$-variate analytic polynomials $f$ of degree at most $d$ and individual degree at most $K-1$, $\|f\|_{L^\infty(\mathbf{D}^n)}\leq C(X)^d\|f\|_{L^\infty(X^n)}$ for any fixed $X$ in the unit disc $\mathbf{D}$ with $|X|=K$.

    The dependence on $d$ in the constant is tight for such small sampling sets, which arise naturally for example when studying polynomials of bounded degree coming from functions on products of cyclic groups.
    As an application we obtain a proof of the cyclic group Bohnenblust--Hille inequality with an explicit constant $\mc O(\log K)^{2d}$.
\end{abstract}

	\subjclass[2020]{41A17, 41A63, 42B05, 32A08}
	
	\keywords{Bernstein discretization inequality, Marcinkiewicz--Zygmund inequality, Polynomial interpolation formula, Bohnenblust--Hille inequality}

	\thanks{
		L.B. is supported by the Collaborative Research Center 1060 funded by the
		Deutsche Forschungsgemeinschaft and the Hausdorff Center for Mathematics,
		funded by the DFG under Germany's Excellence Strategy - GZ 2047/1}
	\thanks{
		O.K. was supported in part by a grant from the Israel Science Foundation (ISF Grant No. 1774/20), and by a grant from the US-Israel Binational Science Foundation and the US National Science Foundation (BSF-NSF Grant No. 2020643}
	\thanks{
		J.S. is supported by Chris Umans' Simons Institute Investigator Grant.}
	\thanks{
		A.V. is supported by NSF DMS-1900286, DMS-2154402 and by Hausdorff Center for Mathematics.}

	\maketitle

\section{Introduction and motivation}
Discretization inequalities control the $L^p$, $1\leq p\leq \infty$ norm of functions $f$ by their $L^p$ norm over some finite sampling set.
In this work we present discretization inequalities in a parameter regime motivated by harmonic analysis on high-dimensional discrete spaces. 
We sketch these motivations now; self-contained statements of our main results are in Section \ref{sect:main-results}.

Exemplary among applications of harmonic analysis to discrete spaces is harmonic analysis on the hypercube $\{-1,1\}^n$, which has led both to new proofs (and sharpenings) of classical results \cite{beckner1975inequalities,bobkov1997isoperimetric}, as well as to many discoveries in combinatorics and theoretical computer science---seminal examples of which are \cite{KKL, FriedgutJunta,FKN}; see \cite{ODbook} for more.
A guiding application for us is the recent introduction of the Bohnenblust--Hille (BH) inequality to statistical learning theory.

The classical BH inequality \cite{BH31} states that for an $n$-variate degree-$d$ analytic polynomial $f$ on the polytorus $\T^n={\{z\in\C:|z|=1\}^n}$,
\begin{equation}
\label{eq:og-BH-sketch}
\|\widehat{f}\|_{\frac{2d}{d+1}}\leq C(d)\|f\|_{\T^n}\,,
\end{equation}
for a constant $C(d)$ that, crucially, is independent of dimension $n$.
Here and throughout the norm $\|\cdot\|_X$ for a space $X$ denotes the supremum norm, and $\widehat{f}$ denotes the sequence of coefficients of $f$.
Eq. \eqref{eq:og-BH-sketch} has a long history: it generalizes Littlewood's celebrated 4/3 inequality \cite{LI30} and derives from Bohnenblust and Hille's resolution \cite{BH31} of the \textit{Bohr's strip problem} \cite{Bohr13} on the convergence of Dirichlet series $\sum_na_nn^{-s}$.
Since then, the BH inequality has become an important tool in harmonic analysis and functional analysis, for example to study Sidon constants and the asymptotic behavior of Bohr radii \cite{DFOCOS,bayart}.
We remark that in these applications it is important to have good control of the BH constants.
See the recent monograph \cite{Defant_García_Maestre_Sevilla-Peris_2019} for more discussion and applications to infinite dimensional holomorphy, probability theory, and analytic number theory.

In view of these developments, it is natural to ask for generalizations of the BH inequality to other groups, such as the discrete hypercube.
The hypercube BH inequality, which was first proved in \cite{Blei_2001} and sharpened in \cite{DMP}, gives the analogous estimate for any $f:\{\pm 1\}^n\to\R$ of degree at most $d$:
\begin{equation}
    \label{ineq:bh-sketch}
    \|\widehat{f}\|_{\frac{2d}{d+1}}\Lesssim{d}\|f\|_{\{\pm 1\}^n}.
\end{equation}
Here and in what follows, $A\lesssim_{d}B$ means $A\le C(d)B$ for a constant $C(d)>0$ depending only on $d$.
The sharp dependence on degree $d$ of the constants in \eqref{eq:og-BH-sketch} and \eqref{ineq:bh-sketch} are longstanding open problems; see \cite{DSP,DMP,Defant_García_Maestre_Sevilla-Peris_2019} for more.

Surprisingly, Eskenazis and Ivanisvili \cite{eskenazis2022learning} recently applied the hypercube BH inequality to obtain exponential improvements in a central task in statistical learning theory called \emph{low-degree learning}: find (with high probability) an $L^2$ approximation to an unknown degree-$d$ polynomial $f:\{\pm 1\}^n\to[-1,1]$ using only the data $\big\{\big(x^{(j)},f(x^{(j)})\big)\big\}_{j}$ for $x^{(j)}$'s drawn independently and uniformly from $\{-1,1\}^n$.

Low-degree learning was introduced in the seminal work of Linial, Mansour, and Nisan \cite{LMN} who gave an algorithm using $\mc O(n^d\log n)$ samples $\big(x^{(j)},f(x^{(j)})\big)$. 
The best dependence on dimension $n$ for this task remained $n^{\Omega(d)}$
until Eskenazis and Ivanisvili \cite{eskenazis2022learning}  obtained the dramatic reduction to $\mc O(\log n\cdot C^{\poly(d)})$ (one should think $d\ll n$).
They applied the hypercube BH inequality to obtain strong Fourier concentration results for bounded low-degree functions, leading to an efficient thresholding approach for estimating $f$ via its Fourier coefficients $\widehat{f}$.
Extensions of this idea to quantum learning theory, with corresponding noncommutative (or \textit{qubit}) variants of the BH inequality \eqref{ineq:bh-sketch}, swiftly followed \cite{HCP, VZ}.

To expand the reach of these applications, it is necessary to extend harmonic analysis results from the hypercube to other discrete spaces.
A primary generalization is the discrete $n$-torus,
\[\Om_K^n:=\{e^{2\pi \iu k/K}:k=0,1,\ldots, K-1\}^n\,,\]
which arises naturally as the ``hypergrid'' in property testing \cite{v010a017, BCS}, is a key setting for studying the hardness of approximation (\textit{e.g.}, the Plurality is Stablest Conjecture of \cite{KKMO}, see also \cite{MOO}), and is crucial for understanding operators on $K$-level qudits in the quantum (or noncommutative) setting \cite{SVZ}.
Already there is a wealth of applications of harmonic analysis to $\Om_K^n$ in combinatorics \cite{ALON1991203, meshulam}, number theory \cite{BakerSchmidt}, and graph theory \cite{ADFS}, for example.

All functions $f:\Om_K^n\to \C$ can be extended to analytic polynomials with individual degree (the maximum degree in any variable) at most $K-1$ using Fourier expansion.
In the spirit of generalizing the hypercube BH inequality \eqref{ineq:bh-sketch} and learning theory applications to $\Omega_K^n$, we shall pursue Fourier-type estimates for functions $f:\Om_K^n\to \C$ whose corresponding polynomials also have \emph{total} degree bounded by $d$.
Note that our applications stipulate the parameter regime $K < d \ll n$.

Attempts to prove inequalities that are standard on the hypercube ($K=2$) meet significant challenges on the discrete $n$-torus for larger $K$.
For example, if one tries to prove the BH inequality on $\Om_K^n$ (even for $K=3$) by repeating the argument that worked for both the polytorus and the hypercube, one quickly encounters trouble, as detailed in \cite[Appendix A]{SVZbh}.

We now sketch a difficulty appearing already for a simpler estimate:
bounding spectral projections of low-degree polynomials.
Concretely, given a polynomial $f:\Om_K^n\to \C$ of total degree at most $d$ and individual degree at most $K-1$, we seek to control its degree-$\ell$ homogeneous part $f_\ell$ as follows:
\begin{equation}
\label{ineq:spectral-proj-sketch}
	\|f_\ell\|_{\Om_K^n}\leq C(d,K)\|f\|_{\Om_K^n}\,.
\end{equation}
This sort of estimate is a typical dimension-free inequality in harmonic analysis.
When the domain is the polytorus ($K=\infty$) this comparison is a trivial Cauchy estimate and bears constant 1.
For the hypercube ($K=2$), this estimate, usually attributed to Figiel \cite[\S 14.6]{milmanSchechtman}, comes fairly easily as well: given $f:\Omega_2^n\to \R$, $\deg(f)\leq d$, and with $x^*\in\Omega_2^n$ a maximizer of $|f_\ell|$, one considers the polynomial $Q(t):=f(tx_1^*,\ldots, tx_n^*)$ for $t\in[-1,1]$ ($f$ is extended to $[-1,1]^n$ as a multiaffine function).
A Markov--Bernstein-type estimate gives
\[\|f_\ell\|_{\Om_2^n}=\frac{|Q^{(\ell)}(0)|}{\ell!}\leq C(d,\ell)\|Q\|_{[-1,1]}\leq C(d,\ell)\|f\|_{[-1,1]^n},\]
with optimal constant $C(d,\ell)\leq(1+\sqrt2)^d$  \cite[Lemma 1.3 (4)]{DMP}.
The final step is to recognize that $\|f\|_{[-1,1]^n}=\|f\|_{\{-1,1\}^n}$ because the extension of $f$ to $[-1,1]^n$ is affine in each coordinate.

However, as soon as $K=3$ it is quite unclear how to proceed.
For example, one could analogize the argument from above, constructing a polynomial $Q(t)$ with $t$ now in the disk $\D$, to obtain via a Cauchy estimate for holomorphic functions
\[\|f_\ell\|_{\Om_K^n}=\frac{|Q^{(\ell)}(0)|}{\ell!}\leq \|Q\|_{\D}=\|Q\|_{\T}\,.\]
Unfortunately, there is no simple way to relate $\|Q\|_{\T}$ to $\|f\|_{\Om_3^n}$.
Of course $\|Q\|_\T\leq \|f\|_{\T^n}$, but then it seems we would need a dimension-free comparison of the form
\begin{equation}
\label{ineq:sketch-comparison}
\|f\|_{\T^n}\Lesssim{d,K}\|f\|_{\Om_K^n}\,.
\end{equation}

Naive attempts to prove \eqref{ineq:sketch-comparison} lead to constants with exponential dependence on dimension $n$.
And that is for good reason, as the inequality \eqref{ineq:sketch-comparison} would be quite strong.
Not only would it give the spectral projection bound \eqref{ineq:spectral-proj-sketch},
but it would also immediately imply the Bohnenblust--Hille inequality for $\Om_K^n$ via:
\begin{equation}
\label{eq:BH-implication}
\|\widehat{f}\|_{\frac{2d}{d+1}}\;\;\overset{\eqref{eq:og-BH-sketch}}{\Lesssim{\mathrlap{d}}}\;\;\; \|f\|_{\T^n}\;\;\overset{\eqref{ineq:sketch-comparison}}{\Lesssim{\mathrlap{d,K}}}\;\;\;\; \|f\|_{\Om_K^n}\,.\end{equation}
\medskip

\noindent A proof of the inequality \eqref{ineq:sketch-comparison} and its generalizations are the subject of this work.

\subsection{\boldmath Main results: Discretizations of \texorpdfstring{$L^\infty$}{L∞} norms by small sets and a new interpolation formula}
\label{sect:main-results}

Our main result is a generalization of \eqref{ineq:sketch-comparison} where $\Omega_K$ is relaxed to any set of points of cardinality $K$ with uniform spacing in the complex unit disk.
\begin{theorem}
	\label{thm design}
	Let $n\geq 1$ and $\go \geq 2$.
	Consider $Y_n=\prod_{j=1}^nZ_j$ for sets $Z_1,Z_2,\ldots, Z_n \subset \D$ such that for all $1\leq j\leq n$ we have $|Z_j| = \go$, and denote by $\eta$ the minimum pairwise distance,
	\begin{equation}\label{uniform sparsity}
		\eta=\min_{1\le j\le n}\min_{z\neq z' \in Z_j} |z - z'| >0.
	\end{equation}
	Then for any analytic polynomial $f:\D^n\to\C$ of degree $d$ and individual degree $\go-1$,
	\begin{equation}
	\label{ineq:remez-plain}
		\|f\|_{\D^n}\leq C(d,\go)\|f\|_{Y_n}.
	\end{equation}
	Here the constant $C(d, \go) := C(d, \go,\eta)=C(\go,\eta)^d$, and $C(\go,\eta)>0$ depends only on $\go$ and $\eta$.
	
	Moreover if all $Z_j=\Om_\go=\{e^{2\pi\iu k/\go}:k=0,1,\ldots, \go-1\}$ then $C(d,\go) \le \big(\mathcal{O}(\log \go)\big)^{2d}$, which proves \eqref{ineq:sketch-comparison}.
	\end{theorem}
    
\noindent The term \emph{individual degree} here means the maximum degree of $f$ viewed as a univariate polynomial in variable $z_j$, $1\leq j\leq n$.
That is, we say $f$ is of degree at most $d$ and individual degree at most $\go-1$ if it can be expressed as 
\begin{equation}\label{eq:polynomial of degree d and local degree N-1}
		f(\z)=\sum_{\bm\alpha\in \{0,1,\dots, \go-1\}^n:|\bm\alpha|\le d}\widehat{f}(\bm\alpha)\z^{\bm\alpha},\qquad \widehat{f}(\bm\alpha)\in \C.
	\end{equation}
The monomial notation in \eqref{eq:polynomial of degree d and local degree N-1} will be used throughout: for any multi-index $\bm\alpha=(\alpha_1,\dots,\alpha_n)\in \Z_{\geq 0}^{n}$ and point $\z=(z_1,\dots, z_n)\in \C^n$ we define
\begin{equation*}
	 |\bm\alpha|:=\sum_{1\le j\le n}\alpha_j,\qquad \textnormal{and} \qquad \z^{\bm\alpha}:=z_1^{\alpha_1}\cdots z_n^{\alpha_n}.
\end{equation*}
Note also that $\eta$ is not completely independent of $K$; there is an upper bound on what $\eta$ is possible because all points are constrained to be in $\D$. 

\begin{rem}
    Theorem \ref{thm design} holds in the real category as well: specialize the inequality to polynomials with real coefficients and choose $Y_n\subset [-1,1]^n$ to be a real sampling set.
    See also Corollary \ref{thm: remezR}.
\end{rem}

In fact, Theorem \ref{thm design} comes as an immediate consequence of a novel polynomial interpolation formula.

\begin{theorem}
    \label{thm:interpolation}
    Let $Y_n$ and $\eta$ be as in Theorem \ref{thm design}.
    Then for all $\bs z\in \D^n$, there exist coefficients $\{c_{\bs \xi}^{(\bs z)}\}_{\bs \xi\in Y_n}\subset \C$ such that for any degree-$d$ polynomial $f$,
    \begin{equation}
    \label{eq:interpthm}
    f(\bs z)=\sum_{\bs \xi \in Y_n}c_{\bs \xi}^{(\bs z)}f(\bs \xi)\,
    \end{equation}
    with $\sum_{\bs \xi}|c_{\bs \xi}^{(\bs z)}|\leq C(d,K)$ for a constant $C(d,K) := C(d,K,\eta)= C(K,\eta)^d$ independent of dimension $n$.
    When $Y_n = \Omega_K^n$, the explicit bound $C(d,K)\leq \big(\mc O(\log K)\big)^{2d}$ holds.
\end{theorem}


In the statement of Theorem \ref{thm:interpolation} we avoid giving explicit forms of the $c_{\bs \xi}^{(\bs z)}$'s as they are complicated (see \eqref{eq:f a_j}) and a significant portion of the proof is devoted to their development.
For now we remark they are naturally expressed in a probabilistic language, a perspective in approximation theory that has been around since Bernstein's probabilistic proof of the Weierstrass approximation theorem \cite{bernstein12} (\textit{n.b.} that Bernstein's proof gives an approximation to continuous, univariate $f$ through a probabilistic argument, whereas we find an exact interpolation formula for multivariate polynomial functions with coefficients having $\ell^1$-norm free of dimension).

\subsection{Theorems \ref{thm design} and \ref{thm:interpolation} in the context of approximation theory}
Let us situate our main results in the context of discretization inequalities in approximation theory, beginning with dimension $n=1$.
Various corollaries and extensions of Theorems \ref{thm design} and \ref{thm:interpolation}, as well as discussions of the optimality of various parameters, are deferred to Section \ref{sect:discuss}.

For $1<p<\infty$, the so-called \emph{Marcinkiewicz--Zygmund inequality} \cite[Chapter X, Theorem (7.5)]{Zygmund} states that for all analytic polynomials $f$ of degree at most $K-1$, one has 
	\begin{equation}\label{ineq:mz classical}
	C_p^{-1}\cdot\tfrac{1}{\go}\!\!\sum_{z\in \Omega_\go}|f(z)|^p
	\;\le\; \int_{\T} |f(z)|^p dz
	\;\le\; C_p\cdot\tfrac{1}{\go}\!\!\sum_{z\in \Omega_\go}|f(z)|^p.
	\end{equation}
	Here $C_p$ is a constant depending only on $p$ (independent of $\go$), and $\T=\{z\in \C:|z|=1\}$ denotes the unit circle.

The inequality \eqref{ineq:mz classical} is an example of a \emph{discretization of the $L^p$-norm}, and integral norm inequalities of this type are usually called \emph{Marcinkiewicz-type theorems}.
At the endpoint $p=\infty$ this type of inequality is often called a \emph{Bernstein-type theorem} or a \emph{discretization of the uniform norm} (see \cite{bernstein31,bernstein32} and \cite[Chapter X, Theorem (7.28)]{Zygmund}).
In our notation, the $p=\infty$ endpoint of \eqref{ineq:mz classical} reads
\begin{equation}\label{ineq:mz p=infty}
\|f\|_{\Omega_\go}\le \|f\|_{\T}\le C(\go)\|f\|_{\Omega_\go}\,.
\end{equation}
 In this $p=\infty$ case (and unlike $1<p<\infty$) we emphasize the right-hand side inequality cannot have constant independent of $K$.
 See for example \cite[Theorem 5]{OS}.
 
 We refer to surveys \cite{DPTT19,KKLT22} and references therein for more historical background about norm discretizations.
Bernstein-type discretization theorems also have some overlap with \emph{discrete Remez-type inequalities}; see \cite{Y,BY,BLMS} for more discussion.
(These are discretizations of the classical \emph{Remez inequality}, which controls the supremum norms of bounded-degree polynomials by their absolute suprema over subsets of positive measure; see for example 
\cite{remez, TikhonovYudistkii}
for more.)

Now let us return to the high-dimensional case, where Theorem \ref{thm design} can be understood as a Bernstein-type discretization inequality for bounded-degree multivariate polynomials in many dimensions $n$.
In this setting there are intricate tradeoffs between the cardinality (and structure) of the sampling set, the constant in the discretization inequality, and the function space to which the estimate applies.
Recently there has been very important progress on understanding the minimum cardinality of sampling sets when one demands a universal constant (independent from any notion of degree or dimension) in the inequality.

In \cite{KKT23}, Kashin, Konyagin, and Temlyakov give a discretization of the uniform norm that applies to any $N$-dimensional subspace of continuous functions on a compact subset of $\R^n$, achieving a universal constant $2$ with a sampling set of cardinality $9^N$.
Moreover, as the authors show, this is essentially the best possible sampling set cardinality for a Bernstein-type discretization inequality at this level of generality.

On the other hand, much smaller sampling sets---again for $L^\infty$ norm discretizations with universal constants---can be had when one fixes the function space to be polynomials of degree at most $d$.
A significant recent work along these lines is \cite{DP}.
Here Dai and Prymak resolved an important problem of Kro\'o \cite{Kroo2011} in real approximation theory by showing there are discretizations of the uniform norm for $n$-variate polynomials of (total) degree at most $d$ over any convex domain in $\R^n$, with universal constant $2$ and a sampling set of cardinality $C_n d^n$ in our notation.\footnote{\emph{N.b.}, in the notation of \cite{DP} it will be $C_d n^d$ where they used $d$ for the dimension and $n$ for the degree.}
When degree $d$ is large in comparison to dimension $n$, this cardinality $C_n d^n$ matches the dimension of the set of such polynomials, and is therefore the best possible.
(\textit{N.b.} our primary interest is in the opposite of their regime, $d\ll n$.)
 
Our motivating application to functions on $\Omega_K^n$---that is, to obtain a comparison
\[\|f\|_{\T^n}\Lesssim{d,K}\|f\|_{\Omega_K^n}\]
for analytic polynomials $f$ of individual degree at most $K-1$ and total degree at most $d$---is in some ways more demanding, and in other ways much more relaxed, than the works above.
On the one hand, the sampling set $\Omega_K^n$ is a fixed product set of small cardinality.
Existing Bernstein-type estimates do not seem to apply in the parameter regime $K<d$, which is the setting dictated by applications to harmonic analysis in the high-dimensional realm of combinatorics, computer science, and learning theory.
On the other hand, we do not require an absolute constant; indeed, as we discuss in Section \ref{sect:discuss}, dependence of the constant on degree $d$ is unavoidable under these constraints.
We also remark that even smaller sampling sets are achievable in our setting, though they are no longer product sets.
This is explained in the next section.

\section{Discussion: corollaries and aspects of optimality}
\label{sect:discuss}

We now remark on several aspects of optimality of Theorem \ref{thm design} and summarize the remaining results of this paper.

\subsection*{Sharp degree-dependence of the constant}
For simplicity we will argue with $Y_n=\Omega_\go^n$.
Consider the univariate inequality
\begin{equation}
\label{ineq:univariate-discrete}
	\|f\|_{\T}\leq C(K) \|f\|_{\Omega_\go}\,
\end{equation}
for polynomials $f$ with degree at most $K-1$.
A Lagrange interpolation argument shows the best constant in \eqref{ineq:univariate-discrete} has $C(K)>1$ for any $K\geq 3$ \cite[Appendix B]{SVZbh}.
Let $g$ be any extremizer of this inequality and put $f(\bs z)=\prod_{j=1}^{d/(K-1)}g(z_j)$, assuming $K-1$ divides $d$ for simplicity.
Then
\[\|f\|_{\T^n}=\big(C(K)\big)^{d/(K-1)}\|f\|_{Y_n}=:D(K)^d\|f\|_{Y_n}\]
which is exponential in $d$.

On the other hand, for this specific construction one may calculate that $D(K)>1$ does not grow in $K$.
It remains an interesting question to determine the optimal $K$-dependence of the constant in \eqref{ineq:remez-plain}.

\begin{question}
    What is the optimal dependence on $K$ in the constant in \eqref{ineq:remez-plain} of Theorem \ref{thm design}?
\end{question}

\medskip

\subsection*{On the cardinality of the sampling set}
\subsubsection*{Minimal cardinality of product sampling sets}
The cardinality of $Y_n$ in Theorem \ref{thm design} is optimal in the following sense.
If the sampling sets are of the \emph{product} form $Y_n=\prod_{1\le j\le n}Z_j$ and one expects \eqref{ineq:remez-plain} to hold at least for polynomials of individual degree at most $\go-1$, then each $Z_j$ must have cardinality at least $\go$ and so $Y_n$ contains at least $\go^n$ points.
If $|Y_n|$ were any smaller, there would exist a $j$ such that $Z_j$ has at most $\go-1$ points, and no such inequality can hold: the polynomial $f_j(\z):=\prod_{\xi\in Z_j}(z_j-\xi )$ is of degree at most $\go-1$ but $\|f_j\|_{Y_n}=0$.

Contrapositively, if the sampling set has a product set structure $Y_n=\prod_{j=1}^nZ_j$ with $|Z_j|=K$, then the individual degree constraint on $f$ is of course necessary.

\subsubsection*{Improvements for non-product sets}
On the other hand, if we remove the product constraint on our sampling set, we can do better.
Indeed, in Section \ref{sect:small design} we show that we may take a ``small" part of $\prod_{j=1}^nZ_j$ and retain a dimension-free constant.

\begin{theorem}
	\label{thm rest design0}
	Let $\go \geq 2$.
	Consider $\{Z_j\subset \D\}_{j\geq 1}$ a sequence of sets such that for all $j\geq 1$ we have $|Z_j| = \go$ and 
	\[
	\eta:=\min_{1\le j\le n}\min_{z\neq z' \in Z_j} |z - z'| >0.
	\]
	Then for any $\varepsilon>0$ one can find a subset $Y_n$ of size at most $C_1(d,\epsilon) (1+\varepsilon)^n$ contained in $\prod_{j=1}^n Z_j$ such that for any analytic polynomial $f:\D^n\to\C$ of degree $d$ and individual degree $\go-1$,
	\begin{equation}
		\|f\|_{\D^n}\leq C_2(d,\go,\eta,\eps)\|f\|_{Y_n},
	\end{equation}
	where  
\[
	C_1(d,\epsilon)\leq \left(\frac{d}{\varepsilon}\right)^{100d}\,.
	\]
	 Furthermore, if $0 < \varepsilon \leq 1/2$, then 
	\[
	C_2(d,\go, \eta, \epsilon)\leq \exp\Big(C_3(d, \go, \eta) \big(\varepsilon^{-1} \log(\varepsilon^{-1})\big)^d\Big)\,,
	\]
	for some constant $C_3(d, \go, \eta)$ depending on $d,\go$ and $\eta$.
	
\end{theorem}

\subsubsection*{Sharp dependence of sampling set cardinality on dimension $n$}
On the other hand, the cardinality of $Y_n$ cannot be sub-exponential in $n$.
It suffices to prove this for $d=1$; the general $d\ge 1$ case follows immediately by definitions and the case $d=1$.

\begin{theorem}\label{thm:no subexp designs}
Suppose that the uniform norm discretization \eqref{ineq:remez-plain}  holds for sampling set $V_n\subset \D^n$ with $d=1,\go=2$; that is,
\begin{equation*}
    \|f\|_{\D^n}\le C_0\|f\|_{V_n}
\end{equation*}
holds for all multi-affine polynomials $f$ of degree $1$ with $C_0>1$ being the best constant, then $|V_n|\ge C_1 C_2^n$, where $C_1>0$ is universal and $C_2>1$ depends on $C_0$.
\end{theorem}
\noindent See Section \ref{sect:sub-exponential} for the proof of Theorem \ref{thm:no subexp designs}.

\subsection*{Uniform separation}%
In Theorem \ref{thm design}, the constant $C(K,\eta)^d$ grows with $\eta^{-1}$, where $\eta$ is the minimum pairwise distance between points in the $Z_j$'s.
In fact, this is unavoidable; uniform separation (\emph{i.e.,} independence of $\eta$ from $n$) is \emph{required} to retain the dimension-freeness of the inequality of Theorem \ref{thm design}.
This is easy to see in one dimension, nor can it be avoided in higher dimensions, as illustrated by the following example.

Suppose $Y_1,Y_2,\ldots $ is a sequence of sets with $Y_n\subset\D^n$ and $c(n)$ is a sequence of coordinates; that is, $1\leq c(n)\leq n$ for all $n$.
Let $P_n=\{z_{c(n)}:\bm z\in Y_n\}\subset \D$ be the projection of $Y_n$ onto the $c(n)$-th coordinate.
Suppose $|P_n|=\go$ for all $n$ and 
\[\lim_{n\to\infty} \min_{z\neq z'\in P_n}|z-z'|=0\,.\]
For each $P_n$ we may then choose a subset $A_n\subset P_n$ with $|A_n|=\go-1$ and an excluded point $\zeta_n^*$ such that $P_n=A_n\sqcup \zeta_n^*$ and
\[\min_{\zeta\in A_n}|\zeta_n^*-\zeta|\leq \eps_n,\]
where $\lim_{n\to\infty}\eps_n=0$. 
Now consider the sequence of polynomials
\[f_n(\z) := \textstyle\prod_{\zeta\in A_n}(z_{c(n)}-\zeta)\,.\]
Certainly $\|f_n\|_{\D^n}$ is at least as large as any of its coefficients, so we have $\|f_n\|_{\D^n}\geq 1$.
On the other hand, $\|f_n\|_{Y_n}$ is very small: $f_n(\bm z)=0$ for all $\bm z\in Y_n$ except those $\bm z$ with $z_{c(n)}=\zeta^*_n$, and for these $\z$ we have
\[|f_n(\ldots, \zeta_n^*,\ldots )|=\big|\textstyle\prod_{\zeta\in A_n}(\zeta_n^*-\zeta)\big|\leq \eps_n\cdot 2^{\go-2},\]
which tends to $0$ as $n\to \infty$.
Therefore no dimension-free uniform norm discretization \eqref{ineq:remez-plain} is available for such $(Y_n)_{n\geq 1}$.
\medskip

\subsection*{Proof ideas of the main result}

Our approach is probabilistic and we sketch it here with $Y_n =\Omega_\go^n$ for simplicity.
In one coordinate, polynomial interpolation admits a probabilistic interpretation of the form
\begin{equation}
\label{eq:idea-one-coord}
f(z)=D\cdot \E [R\cdot f(W)],
\end{equation}
where $D=D(K)>1$ is a constant and $R$ and $W$ are correlated random variables taking values in $\Omega_4$ and $\Omega_\go$ respectively.
Repeating \eqref{eq:idea-one-coord} coordinatewise gives the identity
\begin{equation}
\label{eq:idea-basic-repeat}
	f(\bs z)=D^n\E\left[\left(\textstyle\prod_{j=1}^nR_j\right) f\big(W_1,\ldots, W_n\big)\right],
\end{equation}
which immediately implies a discretization inequality of the desired form, except with exponential dependence on $n$.
The idea is to notice that \eqref{eq:idea-basic-repeat} is an expectation over $n$-many independent pairs of variables $(R_j, W_j)$, while $f$ is of bounded total degree $d$ and thus is not very ``aware'' that $\big((R_1,W_1),\ldots,(R_n,W_n)\big)$ is a product distribution.

It turns out that by introducing certain correlations among the $W_j$'s, we can reduce the power on $D$ at the expense of picking up an error term:
\begin{equation}
\label{eq:idea-m=d}
	f(\bs z)=D^d\E\left[\left(\textstyle\prod_{j=1}^dS_j\right) f\big(\widetilde{W}_1,\ldots, \widetilde{W}_n\big)\right]\;+\;\mathrm{error}_{f,\bs z}\,.
\end{equation}
Here the $S_j$'s are i.i.d. over $\Omega_4$ and the $\widetilde{W}_j$'s are still over $\Omega_\go$, but now the joint distribution $(\widetilde{W}_1,\ldots, \widetilde{W}_n)$ has an intricate dependence structure.
If we only had the first term we would be done of course, and with the right $d$-dependence in the constant.
To remove the error term, we will take advantage of algebraic features of the error's relationship to the introduced correlations.
Specifically, the correlation construction actually defines a family of identities similar to \eqref{eq:idea-m=d} of the form
\[f(\bs z)=D^m\E\left[\left(\textstyle\prod_{j=1}^mS_j^{(m)}\right) f\big(\widetilde{W}_1^{(m)},\ldots, \widetilde{W}_n^{(m)}\big)\right]\;+\;\mathrm{error}_{f,\bs z}\left(\tfrac1m\right),\]
for any integer $m>1$, and where $\mathrm{error}_{f,\bs z}$ is a fixed polynomial in $1/m$ of degree at most $d-1$ and with no constant term.
These properties imply there is an affine combination of these identities for $m=d,d+1,\ldots, 2d-1$ that eliminates the error term:
\begin{equation}
\label{eq:idea-final}
	f(\bs z)\;=\;\sum_{m=d}^{2d-1}a_mf(\bs z)\;=\;\sum_{m=d}^{2d-1}a_mD^m\E\left[\left(\textstyle\prod_{j=1}^mS_j^{(m)}\right) f\big(\widetilde{W}^{(m)}_1,\ldots, \widetilde{W}_n^{(m)}\big)\right],
\end{equation}
and where the absolute sum of the $a_m$'s is suitably small.
This directly gives Theorem \ref{thm design}.

\subsection*{A new polynomial interpolation formula}
All the expectations in \eqref{eq:idea-final} are over finite probability spaces, so we actually have proved a new interpolation formula of the form
\begin{equation}
	\label{eq:interp}
f(\bs z)=\sum_{\bs{\xi}\in Y_n}c_{\bs{\xi}}^{(\bs z)}f(\bs{\xi})\,,
\end{equation}
where $\sum_{\bs{\xi}}|c_{\bs{\xi}}^{(\bs z)}|$ is bounded independent of dimension.
This is Theorem \ref{thm:interpolation}.

Comparing \eqref{eq:interp} to classical multivariate polynomial interpolation formulas, we obtain coefficients with dimension-free absolute sum at the expense of sampling more points than strictly necessary.
As a result the linear combination \eqref{eq:interp} is not unique, and it is interesting to understand whether this flexibility can lead to sharpenings of Theorem \ref{thm design}.

In the full proof, the identity \eqref{eq:interp} appears in detail as Equation \eqref{eq:f a_j}.
We hope this interpolation formula can have future applications and offer as a first example usage a short proof of a dimension-free discretization inequality for $L^p$ norms, $1\leq p <\infty$, as we describe next.

\subsection*{\boldmath Dimension-free discretization for \texorpdfstring{$L^p$}{Lp} norms}
Let $L^p(\T^n)$ and $L^p(\Om_\go^n)$ denote the $L^p$-space with respect to the uniform probability measures on $\T^n$ and $\Omega_\go^n$, respectively.
When $Y_n=\Om_\go^n$, one way to prove a dimension-free $L^p$ discretization inequality for $p<\infty$ would be to use hypercontractivity over $\T^n$ \cite{jansonhyper} and over $\Omega_K^n$ \cite{weissler,junge}.
Hypercontractivity is a workhorse of high dimensional analysis \cite{bakryHyperBook, HuHyperBook} and implies dimension-free $L^2$-$L^p$ comparisons for bounded-degree polynomials when $1\le p<\infty$ (see \cite[Chapter 9.5]{ODbook} and \cite[Chapter 8.4]{DFOCOS, Defant_García_Maestre_Sevilla-Peris_2019} for discussion).
For example, with $2\leq p <\infty$, and $f$ a degree-$d$ function on $\Omega_K^n$, the argument is
\[\|f\|_{L^p(\T^n)}\Lesssim{d,p}\|f\|_{L^2(\T^n)}=\|f\|_{L^2(\Omega_\go^n)}\leq \|f\|_{L^p(\Omega_\go^n)}\,,\]
where hypercontractivity on the polytorus is applied in the first inequality.
Note, however, that such a hypercontractivity argument does not work for $p=\infty$.

In Section \ref{sect:lp} we show a proof that avoids hypercontractivity altogether by making use of the interpolation formula \eqref{eq:interp} (or more concretely, Equation \eqref{eq:f a_j}). 
The main result of Section \ref{sect:lp} is the following.

\begin{theorem}\label{thm:p norm design}
Let $d,n\ge 1, \go\ge 2$. Let $1\leq p\leq \infty$. Then for each polynomial $f:\T^n\to \C$ of degree at most $d$ and individual degree at most $\go-1$, the following holds:
	\[\|f\|_{L^p(\T^n)}\le C(d,\go)\|f\|_{L^p(\Om_\go^n)}\]
	with $C(d,\go) \le d(C_1 \log(\go)+C_2)^{d}$ with universal $C_1, C_2>0$.
\end{theorem}
We remark that the constant in the inequality of Theorem \ref{thm:p norm design} is independent from $p$ but dependent on $d$, so has a different character from Marcinkiewicz--Zygmund inequalities, where the constant depends on $p$ but is typically required to be independent from the total degree $d$ for $1<p<\infty$.

\medskip

\subsection*{Consequences}
It is worthwhile to mention that our results improve various known work. 

As outlined before \eqref{ineq:sketch-comparison}, we may combine Theorem \ref{thm design} with a standard Cauchy estimate to obtain the so-called \emph{Figiel's inequality} for spectral projections of polynomials on $\Om_\go^n$.

\begin{corollary}
    Let $f: \Om_\go^n \to \C$ be a degree-$d$ polynomial with individual degree at most $K-1$.
    For $0 \le \ell \le d$ let $f_\ell$ be the degree-$\ell$ homogeneous part of $f$.
    Then 
    \[
        \|f_\ell\|_{\Omega_\go^n} \le \big(\mc O(\log K)\big)^{2d}\|f\|_{\Omega_\go^n}\,.
    \]
\end{corollary}

\medskip

The chain of inequalities \eqref{eq:BH-implication} shows that Theorem \ref{thm design} relates the best constants $\mathrm{BH}_{\Omega_\go}^{\leq d}$, $\mathrm{BH}_{\T}^{\leq d}$ in the BH inequality for cyclic groups and the polytorus by
\[\mathrm{BH}_{\Omega_\go}^{\leq d} \leq \big(\mathcal{O}(\log \go)\big)^{2d}\cdot \mathrm{BH}_{\T}^{\leq d}\,.\]
In \cite{bayart} the authors obtain the best-known bound for the BH inequality on the polytorus, showing that $\mathrm{BH}_{\T}^{\leq d} \le C^{\sqrt{d \log d}}$. As a consequence the following BH inequality holds for cyclic groups.  

\begin{corollary}
    \label{cor:BH}
    Let $f: \Om_\go^n \to \C$ be a degree-$d$ polynomial with individual degree at most $K-1$. Then
    \begin{equation}\label{ineq: bh cyclic}
		\left(\sum_{|\bm\alpha|\le d}\left|\widehat{f}(\bm\alpha)\right|^{\frac{2d}{d+1}}\right)^{\frac{d+1}{2d}}
		\le \big(\mc{O}(\log \go)\big)^{2d}\|f\|_{\Omega_\go^n},
    \end{equation}
    where $C$ is a universal constant.
\end{corollary}
\noindent This improves upon the constant $C(K)^{d^2}$ previously obtained in \cite{SVZbh} and, as an application, improves the sample complexity of various classical and quantum learning tasks.

Note that the best constant $\mathrm{BH}_{\{\pm 1\}}^{\leq d}$ in the hypercube BH is also known to be at most $C^{\sqrt{d\log d}}$
\cite{DMP}.
This raises a natural question about the best constant for intermediate $K$, $3\leq K<\infty$.

\begin{question}
    What is the best constant for the cyclic-group Bohnenblust--Hille, $\mathrm{BH}_{\Omega_K}^{\leq d}$?
    As a starting point, is $\mathrm{BH}_{\Omega_K}^{\leq d}$ subexponential in $d$?
\end{question}

\medskip

When specialized to real polynomials, our inequality is related to the discrete variants of Remez inequalities of \cite{Y, BY}. 
Theorem \ref{thm design} shows that for certain grids, the constant can be dimension-free (\emph{c.f.,} \cite[Equation (2.5)]{Y}).
We formulate here a real variable version for convenience, in the case of the regular grid. 

\begin{corollary}
	\label{thm: remezR}
	Fix $d,n\ge 1$. Let $f:[-1,1]^n\to \R$ be a real polynomial of degree at most $d$ and individual degree at most $K-1$,
	that is,
	\begin{equation*}
		f(\x)=\sum_{\bm\alpha\in \{0,1,\ldots, K-1\}^n :\,|\bm\alpha|\le d} a_{\bm\alpha} \x^{\bm\alpha},\qquad a_{\bm\alpha}\in \R.
	\end{equation*}
	Then we have 
	\begin{equation}
		\|f\|_{[-1,1]^n}\le C(d,K)\|f\|_{G_K^n},
	\end{equation}
	where  $G_K$ is the grid of $K$ points equi-distributed on $[-1,1]$.
	Here $C(d,K)>0$ is a constant depending on $d$ and $K$ only.
\end{corollary}
\begin{proof}Let us use Theorem \ref{thm design} with $f$ a real polynomial of degree at most $d$ and individual degree at most $K-1$, and sampling set $G_K^n\subset [-1,1]^n$.
\end{proof}

\bigskip

\subsection*{Acknowledgments}  The fourth author is grateful to Alexander Borichev who helped to get rid of some non-essential assumptions.
The authors are very grateful to the referees for useful historical references and valuable comments that greatly improved the presentation of the paper.

\section{Discretizations from product sets}
\label{sect:designs}

Here we prove Theorem \ref{thm:interpolation} and obtain Theorem \ref{thm design} as an immediate consequence.

\subsection{Some preparations}
\label{subsect:preparations}
We start with a lemma that records a standard estimate for inverses of Vandermonde matrices. 

\begin{lemma}\label{lem:vandermonde}
Suppose that $d\ge 1$ and $(x_0,\dots, x_{d-1})\in \C^{d}$ is a vector such that 
\begin{equation}\label{ineq:bdd assumption of vandermonde lem}
\max_{0\le j\le d-1}|x_j|\le M,\qquad \min_{0\le j<k\le d-1}|x_j-x_k|\ge \eta
\end{equation}
with $0<\eta, M<\infty$.
Let $V[x_0,\dots, x_{d-1}]=[a_{jk}]_{j,k=0}^{d-1}$ be the $d$-by-$d$ Vandermonde matrix associated to $(x_0,\dots, x_{d-1})$ with $a_{jk}=x_k^j$. Then its inverse, $V[x_0,\dots, x_{d-1}]^{-1}=[b_{jk}]_{j,k=0}^{d-1}$ satisfies
\begin{equation}
|b_{jk}|\le \frac{M^{d-1-k}\binom{d-1}{k}}{\eta^{d-1}},\qquad 0\le j,k\le d-1.
\end{equation}
\end{lemma}

\begin{proof}
Recall that
\begin{equation}\label{eq:inverse of vandermonde}
b_{jk}=\frac{(-1)^{d-1-k}e_{d-1-k}(\{x_0,\dots, x_{d-1}\}\setminus \{x_j\})}{\prod_{0\le m\le d-1:m\neq j}(x_j-x_m)},
\end{equation}
where $e_k(\{y_1,\dots, y_{d-1}\}):=\sum_{1\le i_1< \cdots< i_k\le d-1}y_{i_1}\cdots y_{i_k}$. Then the desired estimate follows immediately from the assumption \eqref{ineq:bdd assumption of vandermonde lem}. 
\end{proof}

Fix a family of distinct points $\{y_j\}_{j=0,\dots, \go-1}\subset \D$ such that 
	\[\min_{j\neq k}|y_{j}-y_{k}|\ge \eta \,.
	\]
	We denote by $V_{\y}$ the $\go$-by-$\go$ Vandermonde matrix associated to $\y:=(y_0,\dots, y_{\go-1})$ with $V_{\y}=[y^j_k]_{0\le j,k\le \go-1}$. Then $V_{\y}$ is invertible with $V^{-1}_{\y}=[b_{jk}]_{0\le j,k\le \go-1}$ satisfying
	\begin{equation}\label{ineq:vandermonde estimates}
		|b_{jk}|\le \eta^{1-\go}\binom{\go-1}{k},\qquad 0\le j,k\le \go-1
	\end{equation}
	according to Lemma \ref{lem:vandermonde}.
	So for any $ x\in \D$, the system (putting $0^0=1$)
	\begin{equation}
	\label{c-s1} 
	\sum _{k=0}^{\go-1} c_{k}(\y;x) y_k^j=\sum _{k=0}^{\go-1} c_{k}(x) y_k^j = x^j,\qquad 0\le j\le \go-1
	\end{equation}
has a unique solution 
\[(c_0(x),c_1(x),\dots, c_{\go-1}(x))^T=V_{\y}^{-1}(1,x,\dots, x^{\go-1})^T\in \C^\go\]
such that by \eqref{ineq:vandermonde estimates}
\begin{equation}\label{ineq:bound of sum of cj}
	\sum_{j=0}^{\go-1}|c_j(x)|\le \sum_{j=0}^{\go-1}\sum_{k=0}^{\go-1}|b_{jk}| |x|^k\le \frac{\go}{\eta^{\go-1}}\sum_{k=0}^{\go-1}\binom{\go-1}{k}=\go \left(\frac{2}{\eta}\right)^{\go-1}
\end{equation}
uniformly in $x\in \D$.
We remark that $c_{k}(x)=c_{k}(\y;x)$ depends on $\y$, while most of the time we only need the estimate \eqref{ineq:bound of sum of cj} that is uniformly bounded for $\y \in \D^n$. So in the sequel, we omit the dependence on $\y$ whenever no confusion can occur. 

\medskip

We write each complex number $c_k(x)$ in the following form
\begin{equation*}
	c_k(x)=c_k^{(1)}(x)-c_k^{(-1)}(x)+\iu c_k^{(\iu)}(x)-\iu c_k^{(-\iu)}(x),
\end{equation*}
where $c_k^{(s)}(x)=c_k^{(s)}(\y;x)$ are given by
\[
c_k^{(1)}(x)=(\Re c_k(x))_+,\qquad c_k^{(-1)}(x)=(\Re c_k(x))_-, 
\]
and
\[
c_k^{(\iu)}(x)=(\Im c_k(x))_+,\qquad c_k^{(-\iu)}(x)=(\Im c_k(x))_-.
\]

Note that \eqref{c-s1} with $j=0$ becomes $\sum _{k=0}^{\go-1} c_{k}(x)=1$, so that 
\begin{equation}
	\label{Re-c}
	\sum_{k=0}^{\go-1}c_k^{(1)}(x)=	\sum_{k=0}^{\go-1}c_k^{(-1)}(x)+1=:c^{(\Re)}(x)=c^{(\Re)}(\y; x),
	\end{equation}
	and
	\begin{equation}
	\label{Im-c}
		\sum_{k=0}^{\go-1}c_k^{(\iu)}(x)=	\sum_{k=0}^{\go-1}c_k^{(-\iu)}(x)=:c^{(\Im)}(x)=c^{(\Im)}(\y; x).
	\end{equation}
		
Put
	\[
	L= \max_{x\in \D,\y\in \D^n} \left\{c^{(\Re)}(\y;x), c^{(\Im)}(\y;x)\right\}\ge 0.
	\]
	Notice that  by definition of $L$ and \eqref{ineq:bound of sum of cj}
	\begin{equation}
	\label{N-N}
	L\le \max_{x\in \D,\y\in \D^n}\sum_{k=0}^{\go-1}|c_k(\y; x)|\le \go\left(\frac{2}{\eta}\right)^{\go-1}.
	\end{equation}
	So $L$ is a constant depending only on $\go$ and $\eta$.
	
		By definition of $L$, we can choose non-negative $t^{(s)}(x)=t^{(s)}(\y;x), s\in \{\pm 1, \pm \iu\}$ such that
	\begin{equation}	\label{Ct-i0}
		\sum_{k=0}^{\go-1}c_k^{(1)}(x) + t^{(1)}(x)  =L+1,  
	\end{equation}
\begin{equation}	\label{Ct-i}
	\sum_{k=0}^{\go-1}c_k^{(s)}(x) + t^{(s)}(x) =L, \quad s=-1, \iu, -\iu.
\end{equation}
	It follows immediately from \eqref{Re-c}, \eqref{Im-c}, \eqref{Ct-i0}, and \eqref{Ct-i} that
	\begin{equation}
	\label{t-i}
	 t^{(1)}(x)  = t^{(-1)}(x) \qquad  \textnormal{and}\qquad t^{(\iu)}(x)= t^{(-\iu)}(x).
	\end{equation}

\medskip
	
Put $D:=4L+1$. We need certain functions on $[0, D]$. The first one is $r=r(t)$ that depends only on $D$. The second one is $\W^{(x)}(t)=\W_{\y}^{(x)}(t)$ that depends on $x\in \D$ and $\y\in \D^n$. As before, we omit the $\y$-dependence in the notation whenever no confusion can occur.

We first divide the interval $[0,D]=[0,4L+1]$ into the disjoint union of intervals
	\begin{equation*}
		[0,D]=I^{(1)}\cup I^{(-1)}\cup I^{(\iu)}\cup I^{(-\iu)},
	\end{equation*}
with lengths $|I^{(1)}|=L+1$ and $|I^{(-1)}|=|I^{(\iu)}|=|I^{(-\iu)}|=L$. Then the function $r:[0,D]\to \{\pm 1,\pm \iu \}$ is defined as follows:
\begin{equation*}
	r(t)=s,\qquad t\in I^{(s)}.
\end{equation*}
To define $\W^{(x)}$ for any $x\in \D$, recall that \eqref{Ct-i0} and \eqref{Ct-i} allow us to further decompose each $I^{(s)}$ into disjoint unions:
\begin{equation*}
	I^{(s)}=\bigcup_{k=0}^{\go-1}\left(I^{(s)}_{k}(x)\cup J^{(s)}_{k}(x)\right),\qquad s=\pm 1, \pm \iu
\end{equation*}
with (recall again the omitted $\y$-dependence so actually $I^{(s)}_{k}(x)=I^{(s)}_{k}(\y;x)$ and $J^{(s)}_{k}(x)=J^{(s)}_{k}(\y;x)$)
\begin{equation*}
	|I_k^{(s)}(x)|=c_k^{(s)}(x)\qquad\textnormal{and}\qquad	|J_k^{(s)}(x)|=\frac{t^{(s)}(x)}{\go}
\end{equation*}
for all $0\le k\le \go-1$ and $s\in \{\pm 1, \pm \iu\}$.
Then the function $\W^{(x)}=\W^{(x)}_{\y}:[0,D]\to \{y_0,\dots, y_{\go-1}\}$ is defined as
\begin{equation*}
	\W^{(x)}(t)=y_k,\qquad t\in I^{(s)}_{k}(x)\cup J^{(s)}_{k}(x)
\end{equation*}
for all $0\le k\le \go-1$ and $s\in \{\pm 1, \pm \iu\}$.

Now assume that $U$ is a random variable uniformly distributed on $[0,D]$. Then by definition
	\begin{equation}
	\label{eq:expectation of r}
	\bE[ r(U)] = \frac1{D}.
\end{equation}
When $U$ takes values in $I^{(s)}_{k}(x)\cup J^{(s)}_{k}(x)$, we have 
\begin{equation*}
	r(U)=s,\qquad \W^{(x)}(U)=y_k.
\end{equation*}
This, together with definitions of $c_k^{(s)}(x), I^{(s)}_{k}(x), J^{(s)}_{k}(x)$, and equations \eqref{c-s1} and \eqref{t-i}, implies 
\begin{align*}
	\bE\left[r(U)\W^{(x)}(U)^{\alpha}\right]
	&=\frac{1}{D}\sum_{s=\pm 1,\pm \iu}\sum_{k=0}^{\go-1}sy_k^{\alpha}\left(c^{(s)}_k(x)+\frac{t^{(s)}(x)}{\go}\right)\\
	&=\frac{1}{D}\sum_{k=0}^{\go-1}y_k^{\alpha}\sum_{s=\pm 1,\pm \iu}
	s c^{(s)}_k(x)+\frac{1}{D}\sum_{s=\pm 1,\pm \iu}s t^{(s)}(x)\\
	&=\frac{1}{D}\sum_{k=0}^{\go-1}c_k(x)y_k^{\alpha}+0\\
	&=\frac{1}{D}x^{\alpha},
\end{align*}
whenever $0\le \alpha\le \go-1.$ We record this fact for later use: For $U$ uniformly distributed on $[0,D]$, we just proved that
\begin{equation}\label{eq:expectation of product r and w}
	\bE\left[r(U)\W^{(x)}(U)^{\alpha}\right]=\frac{1}{D}x^{\alpha},\qquad x\in \D, \quad0\le \alpha\le \go-1.
\end{equation}

\subsection{The proof idea and difficulty} We fix $\y=(y_0,\dots, y_{\go-1})\in \D^\go, \x=(x_1,\dots, x_n)\in \D^n$ and follow the notations in the previous subsection. We are going to use the functions $r:[0,D]\to \{\pm 1,\pm \iu\}$ and $\W^{(x_j)}=\W^{(x_j)}_{\y}:[0,D]\to \{y_0,\dots, y_{\go-1}\}$ defined above to prove that
\begin{equation*}
	|f(\x)|\le C\max_{\z\in \{y_0,\dots, y_{\go-1}\}^n}|f(\z)|
\end{equation*}
with a dimension-free constant $C$ for certain polynomials $f$. The idea is based on \eqref{eq:expectation of r} and \eqref{eq:expectation of product r and w}. Let us fix $d\ge 1$ and start with a monomial
 \[M(\z)=z_{i_1}^{\al_{i_1}}\dots z_{i_k}^{\al_{i_k}},\qquad \z=(z_1,\dots, z_n)
\]
where  $i_1<\cdots<i_k$ and $1\le \alpha_{i_j}\le \go-1$.
Suppose that $U_1,\dots, U_n$ are i.i.d. random variables uniformly distributed on $[0,D]$. 
Then for any $\supp(\bm\alpha)\subset J\subset [n]:=\{1,\dots, n\}$, we have by \eqref{eq:expectation of r}, \eqref{eq:expectation of product r and w} and the independence that 
\begin{align*}
		&\phantom{{}={}}\bE\left[\prod_{j\in J} r(U_{j})\cdot M\left(\W^{(x_{1})}(U_1),\dots, \W^{(x_{n})}(U_n)\right)\right]\\
		&=\prod_{j=1}^{k}	\bE\left[r(U_{i_j})\W^{(x_{i_j})}(U_{i_j})^{\alpha_j}\right]\prod_{j\in J\setminus \supp(\bm\alpha)}\bE\left[r(U_j)\right]\\
		&=\frac{1}{D^{|J|}}x_{i_1}^{\alpha_1}\cdots x_{i_k}^{\alpha_k}
		=\frac{1}{D^{|J|}}M(\x).
\end{align*}
Recall that $\W^{(x)}$ takes values in $\{y_0,\dots, y_{\go-1}\}$ and $|r(t)|\equiv 1$, and we may then deduce 
\begin{equation*}
	|M(\x)|\le D^{|J|} \max_{\z \in \{y_0, \dots, y_{\go-1}\}^n}|M(\z)|,
\end{equation*}
 as desired. Note that for a single monomial of degree at most $d$, we may choose $J=\supp(\bm\alpha)$ and thus $|J|\le d$. If we consider linear combinations of monomials whose supports are contained in some bounded interval $J$, then the above argument gives the desired estimate with constant $D^{|J|}$ by linearity in $M$ for fixed $J$. However, this is no longer the case if the size of $J$ is large enough; that is, when $|J|$ depends on $n$.
Consider polynomials of degree at most $d=2$ for example. The above argument works for $z_1 z_2+z_2 z_3$ with a choice of $J=\{1,2,3\}$ giving a constant of $D^3$, but for $z_1 z_2+z_2 z_3+\cdots +z_{n-1}z_n$ we need to choose $J=[n]$ which results in a dimension-dependent constant $D^n$.

\subsection{Proof of partial Theorems \ref{thm design} and \ref{thm:interpolation}: product of many copies of a single subset} 
\label{subsect:partial}
 As before, we fix $\y=(y_0,\dots, y_{\go-1})\in \D^\go, \x=(x_1,\dots, x_n)\in \D^n$. We want to show that for each polynomial $f:\D^n\to \C$ of degree at most $d$ and individual degree at most $\go-1$, it holds that
\begin{align*}
|f(\x)|\le C(d,\go)\max_{\z\in \{y_0,\dots, y_{\go-1}\}^n} |f(\z)|.
\end{align*}
As explained in the last subsection, the proof relies on the functions $r,\W^{(x)}$ and we shall overcome the difficulty of unbounded $|\cup_{\bm\alpha}\supp(\bm\alpha)|$  mentioned there as follows. For a monomial $\z^{\bm\alpha}$ of degree at most $d\ge 1$,
 the length of its support $|\supp(\bm\alpha)|$ is at most $d$ as well. Fix $m\ge d$ (one may take $m=d$ for the moment). Consider some map 
 \[
 P: [n] \to [m].
 \]
 Instead of working with polynomials in $\W^{(x_{1})}(U_{1}),\dots, \W^{(x_{n})}(U_{n})$, we shall consider
 \[\W^{(x_{1})}(U_{P(1)}),\dots, \W^{(x_{n})}(U_{P(n)}),\] which may potentially resolve the problems of constants being not dimension-free. However, it may break the independence and we need to estimate the error terms. For this, we make $P$ random as well.
 
Recall that $\y=(y_0,\dots, y_{\go-1})\in \D^\go$ is fixed, with respect to which one defines $\W^{(x)}=\W^{(x)}_{\y}$ function as above for any $x\in \D$. As before, $U_1,\dots,U_n$ are i.i.d. random variables taking values in $[0,D]$ uniformly.

	\begin{prop}
	\label{part1}
	Fix $k\ge 1$. Consider
	\[M(\z)=\z^{\bm\alpha}=z_{i_1}^{\al_{i_1}}\dots z_{i_k}^{\al_{i_k}}\]
	a monomial on $\D^n$ with $1\le \alpha_{i_j}\le \go-1$ and 
	\[\supp(\bm\alpha)=\{1\le i_1<\cdots <i_k\le n\}.\]
	Then for any $\x=(x_1,\dots, x_n)\in \D^n$ there exists a polynomial $p_{\bm\alpha,\bs x, \bs y}$ of degree at most $k-1$ such that $p_{\bm\alpha,\bs x, \bs y}(0)=0$ and the following holds.
	For any $m\ge k$, let $P:[n]\to [m]$ be constructed by choosing for each $j \in [n]$ uniformly at random $P(j) \in [m]$, then we have	
	\begin{equation}\label{eq:key identity monomial}
		M(\x)=D^m\bE_{U, P}\left[\prod_{\ell=1}^{m}r(U_\ell)\cdot \prod_{j=1}^{n}\W^{(x_j)}(U_{P(j)})^{\alpha_{j}}\right]+p_{\bm\alpha,\bs x, \bs y}\left(\frac{1}{m}\right).
	\end{equation}
\end{prop}
	
Before proceeding with the proof, note that by the independence of $U_\ell$'s, for any fixed $P:[n]\to [m]$ we have
\begin{align*}
	\bE_{U}\left[\prod_{\ell=1}^{m}r(U_\ell)\cdot \prod_{j=1}^{n}\W^{(x_j)}(U_{P(j)})^{\alpha_{j}}\right]
	&= \bE_{U}\left[\prod_{\ell=1}^{m}r(U_\ell)\cdot \prod_{j\in\supp(\bm\alpha)}\W^{(x_j)}(U_{P(j)})^{\alpha_{j}}\right]\\
	&=\bE_{U}\left[\prod_{\ell =1}^{m}\left(r(U_\ell) \prod_{j\in\supp(\bm\alpha):P(j)=\ell}\W^{(x_j)}(U_{\ell})^{\alpha_{j}}\right)\right]\\
	&=\prod_{\ell =1}^{m}\bE_{U}\left[r(U_\ell) \prod_{j\in\supp(\bm\alpha):P(j)=\ell}\W^{(x_j)}(U_{\ell})^{\alpha_{j}}\right].
\end{align*}
To estimate each term in the product let us consider the partition of $\supp(\bm\alpha)$. Fix a partition $\mathcal{S}=\{S_j\}$ of $\supp(\bm\alpha)$.
For any $P:[n]\to [m]$ we say $P$ \emph{induces} $\mathcal{S}$ if 
\[\{P^{-1}(\{j\})\cap \supp(\bm\alpha):j\in [m]\}=\mathcal{S}\]
that is
\[\forall j,\quad  |P(S_j)| = 1 \quad \text{and}  \quad \forall j\neq k,\quad P(S_j) \neq P(S_k)\,.\]
	
	Simple combinatorics give that for any partition $\mathcal{S}$ of $\supp(\bm\alpha)$ 
	\begin{equation}
		\label{AS1}
		\Pr[\textnormal{$P:[n]\to [m]$ induces $\mathcal{S}$}]=\frac{m(m-1)\cdots(m-|\mathcal{S}|+1)}{m^{|\supp(\bm\alpha)|}}.
	\end{equation}
	In particular, it can be represented as 
	\begin{equation}
		\label{AS2}
		\Pr[P \textnormal{ induces $\mathcal{S}$}]
		=
		\begin{cases}
			1+q_{|\mathcal{S}|,|\supp(\bm\alpha)|}\big(\tfrac{1}{m}\big) & \text{if } |\mathcal{S}|=|\supp(\bm\alpha)|\\
			q_{|\mathcal{S}|,|\supp(\bm\alpha)|}\big(\tfrac{1}{m}\big) & \text{if } |\mathcal{S}|<|\supp(\bm\alpha)|
		\end{cases}
	\end{equation}
	for some polynomials $q=q_{|\mathcal{S}|,|\supp(\bm\alpha)|}$ with $q(0)=0$ and $\deg(q)<|\supp(\bm\alpha)|$.

\begin{proof}[Proof of Proposition \ref{part1}]
As discussed above, the calculation of expectation in \eqref{eq:key identity monomial} depends on the partition $\mathcal{S}$ of $\supp(\bm\alpha)$. Clearly, $|\mathcal{S}|\le |\supp(\bm\alpha)|=k$. In the special case when $|\mathcal{S}|=|\supp(\bm\alpha)|$, $\mathcal{S}$ is a singleton partition. In this case, we may write $\{j_\ell\}=P^{-1}(\ell),\ell\in [m]$ for $P$ that induces $\mathcal{S}$. For such $P$ we may calculate
\begin{equation*}
	\bE_{U}\left[r(U_\ell) \prod_{j\in\supp(\bm\alpha):P(j)=\ell}\W^{(x_j)}(U_{\ell})^{\alpha_{j}}\right]
	=\bE_{U}\left[r(U_\ell)\right]
	=\frac{1}{D}
\end{equation*}
if $j_\ell\notin \supp(\bm\alpha)$, according to \eqref{eq:expectation of r}; and
\begin{equation*}
	\bE_{U}\left[r(U_\ell) \prod_{j\in\supp(\bm\alpha):P(j)=\ell}\W^{(x_j)}(U_{\ell})^{\alpha_{j}}\right]
		=\bE_{U}\left[r(U_\ell)\W^{(x_{j_\ell})}(U_{\ell})^{\alpha_{j_\ell}}\right]
	=\frac{x_{j_\ell}^{\alpha_{j_\ell}}}{D}
\end{equation*}
if $j_\ell\in \supp(\bm\alpha)$, according to \eqref{eq:expectation of product r and w}.
 All combined, we find 
		\begin{equation}\label{eq:claim1}
	\bE_{U}\left[\prod_{\ell=1}^{m}r(U_{\ell}) \prod_{j=1}^{n}\W^{(x_j)}(U_{P(j)})^{\alpha_{j}}\right] 
	=\prod_{\ell:j_\ell\in \supp(\bm\alpha)}\frac{x_{j_\ell}^{\alpha_{j_\ell}}}{D}\prod_{\ell:j_\ell\notin \supp(\bm\alpha)}\frac{1}{D}
	=\frac{ \x^\alpha}{D^k}.
\end{equation}

\medskip

For general $P$ that induces $\mathcal{S}$ with $|\mathcal{S}|<|\supp(\bm\alpha)|$, it is not easy to calculate 
\begin{equation*}
		\bE_{U}\left[\prod_{\ell=1}^{m}r(U_{\ell})\prod_{j=1}^{n}\W^{(x_j)}(U_{P(j)})^{\alpha_{j}}\right].
\end{equation*}
We claim that there exists a constant $E(\mathcal{S},\bm\alpha,\x,\y)$ independent of $m$ such that 
\begin{equation}\label{eq:claim2}
	\bE_{U}\left[\prod_{\ell=1}^{m}r(U_{\ell})\prod_{j=1}^{n}\W^{(x_j)}(U_{P(j)})^{\alpha_{j}}\right]
	=D^{-m}E(\mathcal{S},\bm\alpha,\x,\y).
\end{equation}
The exact value of $E(\mathcal{S},\bm\alpha,\x,\y)$ is not important; we will find a way to eliminate it afterwards. What matters to us is that it is independent from $m$.

To see the claim, note that for $P$ that induces $\mathcal{S}$ with $|\mathcal{S}|<|\supp(\bm\alpha)|$, 
\begin{align*}
	&\phantom{{}={}}\bE_{U}\left[\prod_{\ell=1}^{m}r(U_{\ell}) \prod_{j=1}^{n}\W^{(x_j)}(U_{P(j)})^{\alpha_{j}}\right]\\
	&=\prod_{\ell =1}^{m}\bE_{U}\left[r(U_\ell) \prod_{j\in \supp(\bm\alpha)\cap P^{-1}(\ell)}\W^{(x_j)}(U_{\ell})^{\alpha_{j}}\right]\\
	&= \frac{1}{D^{m-|\mathcal{S}|}}\underbrace{\prod_{\ell:\supp(\bm\alpha)\cap P^{-1}(\ell)\neq\emptyset}\bE_{U}\left[r(U_\ell) \prod_{j\in \supp(\bm\alpha)\cap P^{-1}(\ell)}\W^{(x_j)}(U_{\ell})^{\alpha_{j}}\right]}_{(*)}
\end{align*}
where we used \eqref{eq:expectation of r} in the last equality. 
We observe that the expectation ($*$) may depend on $(\mathcal{S}, \bm\alpha,\x,\y)$, but does not depend on the specific $P$ inducing $\mathcal{S}$, nor on $m$.
Thus we may define $E(\mathcal{S},\bm\alpha,\x,\y)$ by setting $D^{-|\mathcal{S}|}E(\mathcal{S},\bm\alpha,\x,\y)$ equal to ($*$).

\medskip

According to the above discussion and \eqref{AS2}, we have (we omit the constraint that $\mathcal{S}$ is a partition of $\supp(\bm\alpha)$ in the summation for notational convenience)

\begin{align*}
	&\phantom{{}={}}\bE_{U, P}\left[\prod_{\ell=1}^{m}r(U_{\ell})\prod_{j=1}^{n}\W^{(x_j)}(U_{P(j)})^{\alpha_{j}}\right]\\[0.5em]
	&=\sum_{\mathcal{S} }	\bE_{U, P}\left[\prod_{\ell=1}^{m}r(U_{\ell})\prod_{j=1}^{n}\W^{(x_j)}(U_{P(j)})^{\alpha_{j}}\,\Bigg|\,P\textnormal{ induces $\mathcal{S}$}\right]\Pr[P\textnormal{ induces $\mathcal{S}$}]\\[1em]
	&=\sum_{|\mathcal{S}|=|\supp(\bm\alpha)|} \frac{\x^{\bm\alpha}}{D^m}\Pr[P\textnormal{ induces $\mathcal{S}$}]\\[-0.2em]
	&\qquad \qquad+\sum_{|\mathcal{S}|<|\supp(\bm\alpha)|} \frac{E(\mathcal{S},\bm\alpha,\x,\y)}{D^m} \Pr[\textnormal{P induces $\mathcal{S}$}]\\[1em]
	&=\sum_{|\mathcal{S}|=|\supp(\bm\alpha)|} \frac{\x^{\bm\alpha}}{D^m}	\left[1+q_{|\mathcal{S}|,|\supp(\bm\alpha)|}\big(\tfrac{1}{m}\big) \right]\\[-0.2em]
	&\qquad \qquad+\sum_{|\mathcal{S}|<|\supp(\bm\alpha)|} \frac{E(\mathcal{S},\bm\alpha,\x,\y)}{D^m}	q_{|\mathcal{S}|,|\supp(\bm\alpha)|}\big(\tfrac{1}{m}\big)\\[1em]
	&=\frac{\x^{\bm\alpha}}{D^m}+\frac{1}{D^m}p_{\bm\alpha, \bs x, \bs y}\big(\tfrac{1}{m}\big)
\end{align*}
with
\begin{align*}
p_{\bm\alpha, \bs x, \bs y}(z)
&:=\sum_{|\mathcal{S}|=|\supp(\bm\alpha)|} \x^{\bm\alpha} q_{|\mathcal{S}|,|\supp(\bm\alpha)|}\big(z\big) \\
&\qquad \qquad +\sum_{|\mathcal{S}|<|\supp(\bm\alpha)|} E(\mathcal{S},\bm\alpha,\x,\y)	q_{|\mathcal{S}|,|\supp(\bm\alpha)|}\big(z\big).
\end{align*}
Recalling \eqref{AS2}, $p_{\bm\alpha, \bs x, \bs y}$ satisfies  $p_{\bm\alpha, \bs x, \bs y}(0)=0$ and $\deg(p_{\bm\alpha, \bs x, \bs y})<|\supp (\bm\alpha)|$.
\end{proof}

Proposition \ref{part1} generalizes immediately to low-degree polynomials by linearity.

\begin{prop}\label{thm:key identity}
		Let  $d\ge 1$ and suppose that $f$ is an analytic polynomial on $\D^n$ of degree at most $d$ with individual degree at most $ \go-1$. Then for any $\x=(x_1,\dots, x_n)\in \D^n$, there exists a polynomial $p=p_{f, \bs x, \bs y}$ of degree at most $d-1$ such that $p_{f, \bs x, \bs y}(0)=0$ and the following holds. For any $m\ge d$, let $P:[n]\to [m]$ be constructed by choosing for each $j\in [n]$ uniformly at random $P(j)\in [m]$. We have
	\begin{equation}\label{eq:key identity low degree f}
				f(\x)=D^m\bE_{U, P}\left[\prod_{\ell=1}^{m}r(U_\ell)\cdot f\left(W^{(\x)}(U_P)\right)\right]+p_{f, \bs x, \bs y}\left(\frac{1}{m}\right)\\
	\end{equation}
	where \[W^{(\x)}(U_P):=\left(\W^{(x_1)}(U_{P(1)}),\dots, \W^{(x_n)}(U_{P(n)})\right)\in \{y_0,\dots, y_{\go-1}\}^n = \bs y^n.\]
\end{prop}

\begin{proof}
This follows immediately from \eqref{eq:key identity monomial} by linearity and choosing 
\begin{align*}
p_{f,\bs x, \bs y}(z)
=&\sum_{|\bm\alpha|\le d}\sum_{|\mathcal{S}|=|\supp(\bm\alpha)|} a_{\bm\alpha}\x^{\bm\alpha} q_{|\mathcal{S}|,|\supp(\bm\alpha)|}\big(z\big) \\
&\qquad \qquad +\sum_{|\bm\alpha|\le d}\sum_{|\mathcal{S}|<|\supp(\bm\alpha)|} a_{\bm\alpha}E(\mathcal{S},\bm\alpha,\x,\y)	q_{|\mathcal{S}|,|\supp(\bm\alpha)|}\big(z\big)
\end{align*}
for $f(\bm z)=\sum_{|\bm\alpha|\le d} a_{\bm\alpha}z^{\bm\alpha}$.
\end{proof}

Now we are ready to finish the proof of Theorem \ref{thm:interpolation} when all $Z_j$'s are identical. 

\begin{proof}[Proof of Theorem \ref{thm:interpolation} when $Z_j$'s are identical]
Assume that 
\[Z_j=\{y_0,\dots, y_{\go-1}\}\in \D^{\go}\]
for all $j$, and we shall follow the above notations, \textit{e.g.} $\y=(y_0,\dots, y_{\go-1})$ and $\x=(x_1,\dots, x_n)\in \D^n$ is fixed.
Choose positive integers 
\begin{equation*}
	    d=m_0 < m_1<\dots < m_{d-1}
\end{equation*}
and real numbers $a_0,\dots, a_{d-1}$ in such a way that 
    \begin{equation}
	\label{sys}
	\sum_{j=0}^{d-1} a_j=1 \qquad\text{ and }\qquad\sum_{j=0}^{d-1} \frac{a_j }{m_j^k} =0, \quad k=1, \dots, d-1\,.
\end{equation}

According to Proposition \ref{thm:key identity}, we have for all $m\ge d$ that 
\begin{equation*}
			f(\x)=D^m\bE_{U, P}\left[\prod_{\ell=1}^{m}r(U_\ell)\cdot f\left(W^{(\x)}(U_P)\right)\right]+p_{f, \bs x, \bs y}\left(\frac{1}{m}\right).
\end{equation*}
Applying this identity to $m=m_0,\dots, m_{d-1}$ as above, we get
\begin{align}
f(\x)&=\sum_{j=0}^{d-1}a_{j}f(\x)\nonumber\\
&=\sum_{j=0}^{d-1}a_{j}D^{m_j}\bE_{U, P}\left[\prod_{\ell=1}^{m_j}r(U_\ell)\cdot f\left(W^{(\x)}(U_P)\right)\right]+\sum_{j=0}^{d-1}a_j p_{f, \bs x, \bs y}\left(\frac{1}{m_j}\right)\nonumber\\
&=\sum_{j=0}^{d-1}a_{j}D^{m_j}\bE_{U, P}\left[\prod_{\ell=1}^{m_j}r(U_\ell)\cdot f\left(W^{(\x)}(U_P)\right)\right]\label{eq:f a_j}
\end{align}
where in the first and last equalities we used \eqref{sys}.
Equation \eqref{eq:f a_j} is the explicit form of \eqref{eq:interpthm}.
It remains to instantiate the $m_j$'s and control $\sum_{j=0}^{d-1}|a_{j}|D^{m_j}$.

Let us choose simple $m_j$'s and estimate the $|a_j|$'s.
For this, recall that \eqref{sys} is equivalent to 
\[V(a_0,\dots, a_{d-1})^T=(1,0,\dots, 0)^T\] 
where $V=[a_{jk}]_{j,k=0}^{d-1}$ is the Vandermonde matrix with $a_{jk}=\big(\frac{1}{m_k}\big)^j$. Denoting $[b_{jk}]_{j,k=0}^{d-1}$ the inverse of $V$, then $a_j=b_{j,0}$.
Now let us estimate the absolute sum of $b_{j,0}$'s.

Now we choose $m_j=d+j, 0\le j\le d-1$. Then 
\begin{equation*}
\max_{0\le j\le d-1}|m_j^{-1}|\le \frac{1}{d},\qquad \min_{0\le j<k\le d-1}|m_j^{-1}-m_k^{-1}|\ge \frac{1}{4d^2}.
\end{equation*}
By \eqref{eq:inverse of vandermonde}, we have 
\begin{equation*}
|a_j|=|b_{j,0}|
=\frac{\prod_{0\le k\le d-1:k\neq j}m_k^{-1}}{\prod_{0\le k\le d-1:k\neq j}|m_j^{-1}-m_k^{-1}|}
=\prod_{0\le k\le d-1:k\neq j}\frac{m_j}{|m_k-m_j|}.
\end{equation*}
Recalling $m_j=d+j$, one obtains the estimate
\begin{equation*}
|a_j|=\frac{(d+j)!}{j!(d-1-j)!}
=\frac{d-j}{d+j}\frac{(d+j)^d}{d!}\binom{d}{j}\le C\cdot(2e)^d\left(\frac{d+j}{d}\right)^{d}
\le C\cdot(4e)^d,
\end{equation*}
for all $0\le j\le d-1$ with some universal $C>0$, 
and thus 
\begin{equation}\label{ineq:estimate of a_j D}
\sum_{j=0}^{d-1}|a_{j}|D^{m_j}\le Cd(4eD^2)^d\le C(\go,\eta)^d,
\end{equation}
for some $C(\go,\eta)>0$ depending on $\go$ and $\eta$ only. 
Therefore, we have shown that 
\begin{align*}
|f(\x)|\le C(\go,\eta)^d\max_{\z\in \{y_0,\dots, y_{\go-1}\}^n} |f(\z)|
\end{align*}
for all $\x\in \D^n$. This finishes the proof of Theorem \ref{thm:interpolation} when all the $Z_j$'s are the same. 
\end{proof}

\begin{proof}[Proof of Theorem \ref{thm design} when $Z_j$'s are identical]
Let $\bs x$ maximize $|f|$ on $\D^n$.
Then directly from \eqref{eq:f a_j} we conclude
\begin{align*}
\|f\|_{\D^n}=|f(\x)|&\le \sum_{j=0}^{d-1}|a_{j}|D^{m_j}\cdot\max_{\z\in \{y_0,\dots, y_{\go-1}\}^n} |f(\z)|\\
&\leq C(K,\eta)^d\|f\|_{Y_n}\qedhere\,.
\end{align*}

\end{proof}

  \subsection{\boldmath Improved estimate for groups \texorpdfstring{$\Om_\go^n$}{Ω\_K\string^n}}
\label{OmNn}

Now let us prove the second claims of Theorems \ref{thm design} and \ref{thm:interpolation}, that is, we may choose the constant $C$ in \eqref{ineq:remez-plain} to be $C(d,\go)\le \big(\mathcal{O}(\log(\go)\big)^{2d}$ when all $Z_j=\Om_\go$, $j=1, \dots, n$.
This is based on the following lemma. 

\begin{lemma}
    \label{lem:dft}
    Suppose $z\in \D$, $\om=e^{\frac{2\pi \iu}{\go}}$.
    Then there exists $\bm c:= (c_0,\ldots, c_{\go-1})$ such that for all $k=0,1,\ldots, \go-1$,
    \[z^k=\sum_{j=0}^{\go-1}c_j(\omega^j)^k.\]
    Moreover, $\|\bm c\|_1\leq B\log(\go)$ for a universal constant $B$.
\end{lemma}
\begin{proof}
A direct computation shows that 
\begin{equation}
c_j=\frac{1}{\go}\sum_{k=0}^{\go-1}\omega^{-jk}z^k
=\frac{1-z^\go}{\go(1-\omega^{-j}z)},\qquad 0\le j\le \go-1.
\end{equation}
Since $z\in \D$, we have by triangle inequality that
\begin{equation}
|c_j|\le \min\left\{1, \frac{2}{\go|\omega^j-z|}\right\},\qquad 0\le j\le \go-1.
\end{equation}
By symmetry, we may assume that $\omega^0=1$ is a closest point in $\Omega_\go$ to $z$ and $\omega$ is the next-closest point in $\Omega_\go$ to $z$. It can be that $z=1$, so we use the upper bound $|c_0|\le 1$. We also estimate $|c_1|\le 1$. For $2\le j\le \go-1 $, we use the estimate
$$
\sum_{j=2}^{\go-1}\frac{2}{\go|\omega^j-z|}
\le 2\sum_{j=2}^{\lceil \frac{\go-1}{2}\rceil }\frac{2}{\go|\omega^j-z|}.
$$
For $2\le j\le \lceil \frac{\go-1}{2}\rceil$, we use the estimate
$$
|\omega^j-z|\ge 2\sin\left(\frac{(j-1)\pi}{\go}\right)\ge \frac{4}{\pi}\cdot \frac{(j-1)\pi}{\go}=\frac{4(j-1)}{\go}.
$$
All combined, we get 
\begin{equation}
    \sum_{j=0}^{\go-1}|c_j|
    \le 1+1+2\sum_{j=2}^{\lceil \frac{\go-1}{2}\rceil }\frac{2}{\go|\omega^j-z|}
    \le 2+2\sum_{j=2}^{\lceil \frac{\go-1}{2}\rceil }\frac{2}{4j}
    \le 2+\sum_{j=2}^{\lceil \frac{\go-1}{2}\rceil }\frac{1}{j}.
\end{equation}
Then the proof is finished using the estimate $\sum_{k=1}^{L}1/k\le \log (L)+1$. 
\end{proof}

This gives a better upper bound of $L$, namely $L\le C\log (\go)$ improving \eqref{N-N}. So 
\begin{equation}\label{ineq:improved D}
D=4L+1\le 4C\log (\go)+1
\end{equation}
and we obtain the desired bound following the proof of the previous subsection (see \eqref{ineq:estimate of a_j D}).

\subsection{Proof of Theorems \ref{thm design} and \ref{thm:interpolation}: the general case}
	\label{non-direct-pr}
	
	In this subsection we prove Theorems \ref{thm design} and \ref{thm:interpolation} for general $Y_n=\prod_{\ell=1}^{n}Z_\ell\subset \D^n$ with 
	\[
	Z_{\ell}=\{y_{\ell,0}, \dots, y_{\ell, \go-1}\},\qquad 1\le \ell\le n
	\]
	uniformly separated by $\eta>0$ as stated in the theorem; that is,
	\begin{equation*}
	\min_{1\le \ell \le n}\;\min_{0\le j<k\le \go-1}|y_{\ell, j}-y_{\ell,k}|\ge \eta.
	\end{equation*}
	
	The proof is essentially the same as the case when $Z_\ell$'s are identical. In fact, fix $1\le \ell \le n$ and put $\y_\ell=(y_{\ell,0},\dots, y_{\ell,\go-1})\in \D^\go$. As before, for any $x\in \D$ there exist $c_k(x)=c_k(\y_{\ell},x)$ such that 	\begin{equation}
	\label{c-s} 
	\sum _{k=0}^{\go-1} c_k(\y_{\ell},x) y_{\ell,j}^j = x^j,\quad 0\le j\le \go-1,
	\end{equation}
	and we have the same universal bound 
	\begin{equation*}
	\sum _{k=0}^{\go-1} |c_k(\y_{\ell},x)|\le \go \left(\frac{2}{\eta}\right)^{\go-1}
	\end{equation*}
	as in \eqref{ineq:bound of sum of cj}. This allows us to choose the same $L$ and $D=4L+1$ as in subsection \ref{subsect:preparations} to define the functions $r$ and $\W^{(x)}$. Recall that $r:[0,D]\to \{\pm 1, \pm \iu\}$ is the same as before, since it depends only on $D$.
	The definition of $\W^{(x)}_{\y_{\ell}}:[0,D]\to Z_\ell$ depends on $x\in \D$ and the reference set $Z_\ell$ as well as some order $\y_\ell$. Then we still have the following identity as in \eqref{eq:expectation of product r and w}
\begin{equation}
	\bE\left[r(U)\W^{(x)}_{\y_\ell}(U)^{\alpha}\right]=\frac{1}{D}x^{\alpha},\qquad x\in \D, \quad0\le \alpha\le \go-1
\end{equation}
for all $1\le \ell \le n$. Repeating the previous proof word by word, as in \eqref{eq:key identity low degree f} we arrive at 

	\begin{equation}
				f(\x)=D^m\bE_{U, P}\left[\prod_{\ell=1}^{m}r(U_\ell)\cdot f\left(W_{Y_n}^{(\x)}(U_P)\right)\right]+p_{f,\bs x, Y_n}\left(\frac{1}{m}\right)\\
	\end{equation}
	for any $\x=(x_1,\dots, x_n)\in \D^n$ and $f:\D^n\to \C$ in Theorem \ref{thm:interpolation}, where \[W_{Y_n}^{(\x)}(U_P):=\left(\W_{\y_1}^{(x_1)}(U_{P(1)}),\dots, \W_{\y_n}^{(x_n)}(U_{P(n)})\right)\in Y_n,\]
 $p_{f,\bs x, Y_n}(0)=0$ and $p_{f,\bs x, Y_n}$ is of degree at most $d-1$ as before.
    Here the notation $W_{Y_n}^{(\x)}$ might not be a perfect choice as it depends on the some orders of each $Z_\ell$, as indicated in its definition.
    But we abuse the notation here since we only need the fact that it takes values in $Y_n$.
	The rest is the same as the end of subsection \ref{subsect:partial} and thus we finish the proof of Theorem \ref{thm:interpolation}.
    The general case of Theorem \ref{thm design} follows immediately.

	
		\section{Small sampling sets}
\label{sect:small design}

In this section we show sampling sets of cardinality $C(1 + \varepsilon)^n$ exist for arbitrary $\varepsilon > 0$.
Denote by $C(d,\go,\eta)$ the constant obtained in Theorem \ref{thm design}.

\begin{proof}[Proof of Theorem \ref{thm rest design0}]
For any $n\ge 1$, let $k \geq 1$ be an integer such that 
\begin{equation}
    \label{eq k condition}
    ((d+1)k^d)^{1/k} \leq 1 + \varepsilon\,.
\end{equation}
We temporarily assume first that $n$ is a multiple of $k$, $n=n'k$. The reader should think that $n'$ is very large, and $k$ is of the order $\frac{d}{\varepsilon}\log\frac{d}{\varepsilon}$ which is not too large. We are going to choose the desired sampling set in the following way. Write 
\[
\prod_{j=1}^{n}Z_j=\prod_{j=0}^{n'-1}Z_{jk+1}\times \cdots \times Z_{jk+k},
\]
and for each $0\le j\le n'-1$, we are going to choose $M$ (to be defined later) points in $Z_{jk+1}\times \cdots \times Z_{jk+k}$. 

More precisely, define (denoting $\|\bm\alpha\|_p$ the $\ell^p$-norm of a vector $\bm\alpha$)
\begin{equation}
\label{Lambda}
    \Lambda= \{\bm\alpha \in \N^k\,\colon \|\bm\alpha\|_1 \leq d, \|\bm\alpha\|_\infty \leq \go-1\}\,,
\end{equation}
so that $\{\z^{\bm\alpha}\,: \, \bm\alpha \in \Lambda\}$ is a $\mathbb{C}$-vector space basis of the polynomials of degree at most $d$ and individual degree at most $\go-1$ in $k$ variables $z_1, \dotsc, z_k\in \D$. The cardinality of $\Lambda$ is
\begin{equation}
    \label{eq M size}
    M := |\Lambda| \leq (d+1) k^d\,.
\end{equation}
Let $\Y \in (\D^k)^M$. We will write 
\[
    \Y = (\y_1, \dotsc, \y_M)\,, \quad  \y_j = (y_{j 1}, y_{j 2}, \dotsc, y_{j k})\in \D^k\,.
\]
For any $\z =(z_1, \dots, z_k)\in \D^k$, we want to solve the system of $M$ equations
\begin{equation}
    \label{eq system}
    \z^{\bm\alpha}=z_1^{\al_1}\dots z_k^{\al_k} = \sum_{j = 1}^M c_{j}(\z) \y_j^{\bm\alpha}=\sum_{j = 1}^M c_{j}(\z) \prod_{r=1}^k y_{jr}^{\alpha_r} \,, \quad \bm\alpha \in \Lambda
\end{equation}
where $c_j(\z)=c_j(\Y; \z)$ depends on $\Y$ and $\z$ if it exists. 
The system is solvable for all $\z \in \D^k$ whenever the determinant
\begin{equation}
    \label{eq P def}
    P(\Y) := \det [\y_{j}^{\bm\alpha}]_{\substack{1\le j\le M, \bm\alpha \in \Lambda}}
    =\det [\y_{j}^{\bm\beta_i}]_{\substack{1\le i,j\le M}}
    =\sum_{\sigma \in S_M} \sgn(\sigma) \prod_{j =1}^M \y_{j}^{\bm\beta_{\sigma(j)}}
\end{equation}
does not vanish,
where $\{\bm\beta_i:1\le i\le M\}$ is some reordering of $\{\bm\alpha: \bm\alpha\in \Lambda\}$, and $S_M$ denotes the permutation group on $M$ letters. If this is the case, we can further explicitly express the solutions of \eqref{eq system} using Cramer's rule as
\[
    c_j(\z)=c_j(\Y; \z)= \frac{P_j(\Y, \z)}{P(\Y)}\,,
\]
where $P_j(\Y, \z)$ is the determinant of the $M\times M$ matrix obtained by replacing the $j$-th column of $[\y_{j}^{\bm\beta_i}]_{\substack{1\le i, j\le M}}$ by $(\z^{\bm\alpha})_{\bm\alpha\in \Lambda}=(\z^{\bm\beta_i})_{1\le i\le M}$. All the coefficients of this matrix are contained in the unit disc, so by a theorem of Hadamard \cite{Hadamard1893} we have
\[
    |P_j(\Y,\z)| \leq M^{M/2}\,.
\]
On the other hand, the monomials in \eqref{eq P def} for different $\sigma$ are all different, hence using orthogonality of monomials in $L^2(\T^{kM})$ we obtain 
\[
    \int_{\Y \in \T^{kM}} |P(\Y)|^2 \, \mathrm{d}\Y = M!\,.
\]
Here we equip $\T^{kM}$ with the uniform probability measure $d\Y$. 
Therefore, for this analytic polynomial $P$ over $\D^{kM}$ of degree at most $dM$ and individual degree at most $\go-1$, we have 
\begin{equation*}
\|P\|_{\D^{kM}}=\|P\|_{\T^{kM}}\ge \|P\|_{L^2(\T^{kM},d\Y)}=\sqrt{M!}.
\end{equation*}
By Theorem \ref{thm design}, for any $1\le i_1, \dotsc, i_k\le n$ and the sets $Z_{i_1}, \dotsc, Z_{i_k}$, there exists 
$\underline{\Y} \in (Z_{i_1} \times \dotsb \times Z_{i_k})^M$ with
\begin{equation}
    \label{eq Y condition}
    |P(\underline{\Y} )| \geq C(dM,\go,\eta)^{-1} \sqrt{M!}>0\,.
\end{equation}
For $\underline{\Y}$ satisfying \eqref{eq Y condition}, the above system \eqref{eq system} is solvable and we can estimate the size of the solutions $c_j(\underline{\Y};\z)$ of \eqref{eq system}:
\begin{align}
    |c_j(\underline{\Y};\z)|  &\leq \frac{|P_j(\underline{\Y}, z)|}{|P(\underline{\Y})|} \leq C(dM,\go,\eta) \frac{M^{M/2}}{\sqrt{M!}}                  \notag
    \\
    &\leq C(\go,\eta)^{dM} \frac{M^{M/2}}{\left(M/e\right)^{M/2}} \le C(\go,\eta)^{dM}e^{M/2}.
    \label{eq cj est}
\end{align} 
Note that $c_j(\underline{\Y};\z)$ depends on $\underline{\Y}$ and $\z$ but this upper bound of $|c_j(\underline{\Y};\z)|$ does not.

We now construct our sampling set. Recall that $n = n'k$ is a multiple of $k$. For each $s = 0, \dotsc, n'-1$, consider a tuple of points 
\[
    Y^{(s)} = (\y_{1}^{(s)}, \dotsc, \y_{M}^{(s)}) \in (Z_{sk+1} \times \dotsb \times Z_{sk + k})^M
\]
such that \eqref{eq Y condition} holds for $Z_{i_j} = Z_{sk +j}, j=1, \dots, k$.  These are the $M$ points that we choose in $Z_{sk+1} \times \dotsb \times Z_{sk + k}$. Our sampling set
will be 
\[
    \mathbf{Y} \coloneqq \prod_{s=0}^{n'-1} Y^{(s)}\,,
\]
which satisfies, by \eqref{eq k condition} and \eqref{eq M size},
\[
    |\mathbf{Y}| \leq M^{n'} \leq ((d+1)k^d)^{n/k} \leq (1 + \varepsilon)^n\,.
\]

It remains to show that $\mathbf{Y}$ is a sampling set with dimension-free constant.
The argument is essentially the same as in the proof of Theorem \ref{thm design}; the only difference is that the functions $\W$ take values in $Y^{(s)}\subset\D^{kM}$. 
Let $f(\z)=\sum_{\bm\alpha}a_{\bm\alpha}\z^{\bm\alpha}$ be a polynomial in $n$ variables of degree at most $d$ and individual degree at most $\go-1$.
Fix 
\[
\z =(\z_0,\dots, \z_{n'-1})\in \D^{n}=(\D^{k})^{n'} \qquad \textnormal{with} \qquad \z_{s}\in \D^{k}
\]
and $\bm\alpha$ such that $a_{\bm\alpha}\neq 0$. Then $\z^{\bm\alpha}=\prod_{s=0}^{n'-1}\z_{s}^{\bm\alpha_s}$.

\medskip

Since each $Y^{(s)}$ satisfies \eqref{eq Y condition}, the determinant of the system \eqref{eq system} does not vanish.
So we have
\[
    \z_s^{\bm\alpha_s} = \sum_{j = 1}^M c_{j} (\z_{s}) (\y_{j}^{(s)})^{\bm\alpha_s},\qquad \boldsymbol{\alpha} _{s}\in \Lambda
\]
and by \eqref{eq cj est} 
\begin{equation}
    \label{eq L def}
    \sum_{j = 1}^M |c_{j} (\z_{s}) | \leq M C(dM,\go,\eta) e^{M/2} \eqqcolon L\,.
\end{equation}
As before, $L$ is independent of $Y^{(s)}$ and $\z$.

Now  for the convenience of the reader we sketch the construction in Section \ref{sect:designs} with minor modifications. 
Again, we write each complex number $c_{j} (\z_{s})$ in the following form
\begin{equation*}
	c_{j} (\z_{s})=c_{j}^{(1)}(\z_{s})-c_{j}^{(-1)} (\z_{s})+\iu c_{j}^{(\iu)} (\z_{s})-\iu c_{j}^{(-\iu)} (\z_{s}),
\end{equation*}
where all $c_{j}^{(r)} (\z_{s})$'s are non-negative such that 
\[
c_{j}^{(1)} (\z_{s})=(\Re c_{j} (\z_{s}))_+,\qquad c_{j}^{(-1)}(\z_{s})=(\Re c_{j} (\z_{s}))_-, 
\]
and
\[
c_{j}^{(\iu)} (\z_{s})=(\Im c_{j} (\z_{s}))_+,\qquad c_{j}^{(-\iu)} (\z_{s})=(\Im c_{j} (\z_{s}))_-.
\]
Same as before, we may choose non-negative $t^{(1)}(\z_{s})= t^{(-1)}(\z_{s}), t^{(\iu)}(\z_{s})=  t^{(-\iu)}(\z_{s})$ in such a way that
	\begin{equation}	\label{Ct-i0 small design}
		\sum_{j=1}^{M}c_{j}^{(1)} (\z_{s})+ t^{(1)}(\z_{s})  =L+1,  
	\end{equation}
\begin{equation}	\label{Ct-i small design}
	\sum_{j=1}^{M}c_{j}^{(r)} (\z_{s})+ t^{(r)}(\z_{s}) =L, \quad r=-1, \iu, -\iu.
\end{equation}
Moreover, 
	\begin{equation}
	\label{t-i small design}
	 t^{(1)}(\z_{s})  = t^{(-1)}(\z_{s}) \qquad  \textnormal{and}\qquad t^{(\iu)}(\z_{s})= t^{(-\iu)}(\z_{s}).
	\end{equation}

\medskip
	
Again, we put $D=4L+1$ and divide the interval $[0,D]$ into the disjoint union 
	\begin{equation*}
		[0,D]=I^{(1)}\cup I^{(-1)}\cup I^{(\iu)}\cup I^{(-\iu)},
	\end{equation*}
with $|I^{(1)}|=L+1$ and $|I^{(-1)}|=|I^{(\iu)}|=|I^{(-\iu)}|=L$. Define the function $r:[0,D]\to \{\pm 1,\pm \iu \}$ as before
\begin{equation*}
	r(t)=s,\qquad t\in I^{(s)}.
\end{equation*}

For each $\z_{s}$ we further decompose each $I^{(r)}$ into disjoint unions:
\begin{equation*}
	I^{(r)}=\bigcup_{k=1}^{M}\left(I^{(r)}_{k}(\z_{s})\cup J^{(r)}_{k}(\z_{s})\right),\qquad r=\pm 1, \pm \iu
\end{equation*}
with 
\begin{equation*}
	|I_k^{(r)}(\z_s)|=c_{k}^{(r)} (\z_{s})\qquad\textnormal{and}\qquad	|J_k^{(r)}(\z_s)|=\frac{t^{(r)}(\z_{s})}{\go},\qquad 1\le k\le M.
\end{equation*}
Then we define the function $\W^{(\z_s)}_{Y^{(s)}}:[0,D]\to Y^{(s)}$ as
\begin{equation*}
	\W^{(\z_s)}_{Y^{(s)}}(t)=\y^{(s)}_k,\qquad t\in I^{(r)}_{k}(\z_s)\cup J^{(r)}_{k}(\z_s)
\end{equation*}
for all $1\le k\le M$ and $r\in \{\pm 1,\pm \iu\}.$

Suppose $U$ is a random variable uniformly distributed on $[0,D]$, then
\begin{equation}
	\label{eq:expectation of r small design}
	\bE[ r(U)] = \frac1{D},
\end{equation}
and
\begin{align*}
	\bE\left[r(U)\W^{(\z_s)}_{Y^{(s)}}(U)^{\bm\alpha_s}\right]=\frac{1}{D}\z_s^{\bm\alpha_s},\qquad \boldsymbol{z}_{s}\in \mathbf{D}^k, \boldsymbol{\alpha} _{s}\in \Lambda.
\end{align*}

Arguing as before, we may deduce that there exists a polynomial $p=p_{f,\bs z, \mathbf Y}$ of degree at most $d-1$ such that $p_{f,\bs z, \mathbf Y}(0)=0$ and the following holds. For any $m\ge d$, let $P:[n']\to [m]$ be constructed by choosing for each $i\in [n']$ uniformly at random $P(i)\in [m]$. We have
	\begin{equation}
				f(\z)=D^m\bE_{U, P}\left[\prod_{j=1}^{m}r(U_j)\cdot f\left(W_{\Y}^{(\z)}(U_P)\right)\right]+p_{f, \bs z, \mathbf Y}\left(\frac{1}{m}\right)\\
	\end{equation}
	where 
	\[W_{\Y}^{(\z)}(U_P):=\left(\W_{Y_{0}}^{(\z_0)}(U_{P(1)}),\dots, \W_{Y_{n'-1}}^{(\z_{n'-1})}(U_{P(n')})\right)\in \Y,
	\]
 $p_{f,\bs z, \mathbf{Y}}(0)=0$ and $p_{f,\bs z, \mathbf{Y}}$ is of degree at most $d-1$ as before.  

The rest of proof is exactly as before. 
Thus $\mathbf{Y}$ discretizes the uniform norm with constant at most 
\begin{equation}
    \label{eq CL}
    C d (4eD^2)^d\,.
\end{equation}
This completes the proof when $n$ is a multiple of $k$. 

In general, if $n'k<n< (n'+1)k\le n+k-1$, then we may project the sampling set from $\D^{(n' + 1)k}$ onto $\D^{n}$ so that there exists a sampling set $(Y_n)$ of size at most 
\begin{equation}
    \label{eq c'}
    |Y_n|\le (1+\varepsilon)^{n + k - 1} = C_1(d,\varepsilon) (1 + \varepsilon)^n\,.
\end{equation}
Finally, let us compute the precise dependence of $C_1(d,\epsilon)$ and $C_2(d,\go,\eta,\epsilon)$ on $\varepsilon$. 
The condition \eqref{eq k condition} on $k$ holds for all $d \geq 1$ and $0 < \varepsilon \leq 1/2$ if $k$ is the largest integer smaller than 
\[
    100 \frac{d}{\varepsilon} \log\left(\frac{d}{\varepsilon}\right)\,.
\]
This yields for $C_1(d,\varepsilon)$, using \eqref{eq c'}
\[
    C_1(d,\varepsilon) \leq (1 + \varepsilon)^{k} \leq \exp( 100 d \log\left(\frac{d}{\varepsilon}\right)) = \left(\frac{d}{\varepsilon}\right)^{100d}\,.
\]
For $C_2(d,\go,\eta,\varepsilon)$, we obtain with \eqref{eq M size}, \eqref{eq L def} and \eqref{eq CL} the bound
\[
    C_2(d,\go,\eta,\varepsilon)\leq \exp(C(d, \go, \eta) (\varepsilon^{-1} \log(\varepsilon^{-1}))^d)\,.
\]
The number $100$ in the exponent can be brought arbitrarily close to $1$, if $\varepsilon$ is assumed to be sufficiently small.
\end{proof}

\section{Necessity of exponential-cardinality sampling sets}
\label{sect:sub-exponential}

Now we prove Theorem \ref{thm:no subexp designs} which says sampling sets with cardinality sub-exponential in dimension cannot exist.

\begin{proof}[Proof of Theorem \ref{thm:no subexp designs}]
For any $\bm\eps=(\eps_1,\dots, \eps_n)\in \{-1,1\}^n$, consider the polynomials $f_{\bm\eps}(\x)=\sum_{j=1}^{n}\eps_jx_j$ on $\{-1,1\}^n$ of degree at most $1$. Then by definition, 
 \begin{equation*}
 n=\|f_{\bm\eps}\|_{\{\pm 1\}^n}\le C \|f_{\bm\eps}\|_{V_n}. 
 \end{equation*}
In other words, we have for all $\bm\eps\in \{-1,1\}^n$ that 
\begin{equation*}
\max_{\v=(v_1,\dots, v_n)\in V_n}\left|\sum_{j=1}^{n} v_j \eps_j\right|\ge 2\delta n\qquad \textnormal{with}\qquad \delta=\frac{1}{2C}\in (0,\infty).
\end{equation*}
So we have the inclusion $\{-1,1\}^n\subset \bigcup_{\v\in V_n}\Lambda_{\v}$, where 
\[\Lambda_{\v}:=\left\{\bm\eps\in \{-1,1\}^n:|f_{\v}(\bm\eps)|=\left|\sum_{j=1}^{n} v_j \eps_j\right|\ge 2\delta n\right\},\qquad \v\in V_n.\]
For each $\v\in V_n$ and i.i.d. Bernoulli random variables $\eps_1,\dots, \eps_n$, we have by Hoeffding's inequality that 
\begin{align*}
|\Lambda_{\v}|
&=2^n \textbf{Pr}\left[\left|\sum_{j=1}^{n} v_j \eps_j\right|\ge 2\delta n\right]\\
&\le 2^n \textbf{Pr}\left[\left|\sum_{j=1}^{n} \Re(v_j) \eps_j\right|\ge \delta n\right]+2^n \textbf{Pr}\left[\left|\sum_{j=1}^{n} \Im(v_j) \eps_j\right|\ge \delta n\right]\\
&\le 2^{n+1}\exp\left(-\frac{\delta^2 n^2}{2\|\a\|_2^2}\right)+2^{n+1}\exp\left(-\frac{\delta^2 n^2}{2\|\b\|_2^2}\right),
\end{align*}
where $\a=\Re \v$ and $\b=\Im \v$ are real vectors. Recalling that $\v\in V_n\subset \D^n$, we have $\|\a\|_2^2\le n$ and $\|\b\|_2^2\le n$. Therefore, 
\begin{equation*}
|\Lambda_{\v}|\le 2^{n+2}\exp\left(-\delta^2 n/2\right)=4 \left(2e^{-\delta^2/2}\right)^n.
\end{equation*}
All combined, we just proved
\begin{equation*}
2^n=|\{-1,1\}^n|\le \sum_{\v\in V_n}|\Lambda_{\v}|\le 4 |V_n| \left(2e^{-\delta^2/2}\right)^n.
\end{equation*}
This gives the bound 
\begin{equation*}
|V_n| \ge \frac{1}{4}e^{\frac{\delta^2 n}{2}},
\end{equation*}
as desired. 
\end{proof}

\section{ Discretizations of \texorpdfstring{$L^p$}{Lp} norms}
\label{sect:lp}
In this section, we prove Theorem \ref{thm:p norm design} about an $L^p$ version of dimension-free discretizations for products of cyclic groups.

\begin{proof}[Proof of Theorem \ref{thm:p norm design}]
	Write $Y_n=\prod_{j=1}^{n}Z_j$ with $Z_j=\Omega_\go$ and $\y=(1,\omega,\dots, \omega^{\go-1})$ with $\omega=e^{\frac{2\pi \iu}{\go}}.$ According to \eqref{eq:f a_j}, we have for all $\x=(x_1,\dots, x_n)\in \T^n$ that
	\begin{equation*}
	f(\x)=\sum_{j=0}^{d-1}a_{j}D^{m_j}\bE_{U, P}\left[\prod_{\ell=1}^{m_j}r(U_\ell)\cdot f\left(W_{Y_n}^{(\x)}(U_P)\right)\right]
	\end{equation*}
	where $a_j, m_j, 0\le j\le d-1$ are as defined in subsection \ref{subsect:partial} and
	\[W_{Y_n}^{(\x)}(U_P):=\left(\W_{\y}^{(x_1)}(U_{P(1)}),\dots, \W^{(x_n)}_{\y}(U_{P(n)})\right)\in Y_n=\Omega_\go^n.\]
	Again, we abuse the notation $W_{Y_n}^{(\x)}$ which actually depends on some order of each $Z_j$ as specified in its definition. See subsection \ref{non-direct-pr}.
	For any $\bs\xi=(\xi_1,\dots, \xi_n) \in \Om_\go^n$, denote $\x \bs\xi=(x_1\xi_1,\dots, x_n \xi_n)$ the multiplication of $\x$ and $\bs\xi$ as elements in the group $\T^n$. Then  $\x \bs\xi\in \T^n$ and 
	\begin{equation*}
	f(\x  \bs\xi)=\sum_{j=0}^{d-1}a_{j}D^{m_j}\bE_{U, P}\left[\prod_{\ell=1}^{m_j}r(U_\ell)\cdot f\left(W_{Y_n}^{(\x \bs\xi)}(U_P)\right)\right].
	\end{equation*}
	For any $0\le j\le \go-1$, recall that $\omega^{(x_j \xi_j)}_{\y}$ takes values in $\Omega_\go$ with
	\begin{align*}
	\Pr\left[\omega^{(x_j \xi_j)}_{\y}(U)=\omega^k\right]
	&= \frac{1}{D}\sum_{s=\pm 1, \pm \iu}\Big(c_k^{(s)}(\y;x_j \xi_j)+\frac{t^{(s)}(\y;x_j\xi_j)}{\go}\Big)\\
	&= \frac{1}{\go}+\frac{1}{D}\Big\{\sum_{s= \pm 1, \pm \iu}c_k^{(s)}(\y;x_j \xi_j)-\frac{1}{\go}\sum_{\ell=0}^{\go-1}c^{(s)}_\ell(\y;x_j \xi_j)\Big\},
	\end{align*}
    where we used \eqref{Ct-i0} and \eqref{Ct-i}.
Also, $(c_k(\y;x_j\xi_j))_k$ is uniquely determined by 
\begin{equation*}
\sum_{k=0}^{\go-1}c_k(\y;x_j\xi_j) \omega^{k\ell}=(x_j\xi_j)^\ell, \qquad 0\le \ell\le \go-1,
\end{equation*}
or equivalently
\begin{equation*}
\sum_{k=0}^{\go-1}c_k(\y;x_j\xi_j) (\omega^k\overline{\xi_j})^{\ell}=x_j^\ell, \qquad 0\le \ell\le \go-1.
\end{equation*}
So $c_k(\y;x_j\xi_j)=c_k(\y_j;x_j)$ with 
$\y_j=(\overline{\xi_j},\overline{\xi_j}\omega,\dots, \overline{\xi_j}\omega^{\go-1})$ that is the same as $\Omega_\go$ as a set. Here we used the group structure of $\Omega_\go$. Therefore, 
	\begin{align*}
	\Pr\left[\omega^{(x_j \xi_j)}_{\y}(U)=\omega^k\right]
	&=\frac{1}{\go}+\frac{1}{D}\Big\{\sum_{s= \pm 1, \pm \iu}c_k^{(s)}(\y;x_j \xi_j)-\frac{1}{\go}\sum_{\ell=0}^{\go-1}c^{(s)}_\ell(\y;x_j \xi_j)\Big\}\\
	&=\frac{1}{\go}+\frac{1}{D}\Big\{\sum_{s= \pm 1, \pm \iu}c_k^{(s)}(\y_j;x_j)-\frac{1}{\go}\sum_{\ell=0}^{\go-1}c^{(s)}_\ell(\y_j;x_j)\Big\}\\
	&=\Pr\left[\omega^{(x_j)}_{\y_j}(U)=\overline{\xi_j}\omega^k\right].
	\end{align*}
So we just argued that $\omega^{(x_j \xi_j)}_{\y}(U)=\xi_j\omega_{\y_j}^{(x_j)}(U)$. Thus we find 
	\begin{align*}
	f(\x  \bs\xi)
	&=\sum_{j=0}^{d-1}a_{j}D^{m_j}\bE_{U, P}\left[\prod_{\ell=1}^{m_j}r(U_\ell)\cdot f\left(W_{Y_n}^{(\x \bs\xi)}(U_P)\right)\right]\\
	&=\sum_{j=0}^{d-1}a_{j}D^{m_j}\bE_{U, P}\left[\prod_{\ell=1}^{m_j}r(U_\ell)\cdot f\left(\bm\xi W_{Y_n'}^{(\x)}(U_P)\right)\right]
	\end{align*}
	where $Y_n'=Y_n$ with a different ordering of each $Z_j$:
	\[W_{Y_n'}^{(\x)}(U_P)=\left(\W_{\y_1}^{(x_1)}(U_{P(1)}),\dots, \W^{(x_n)}_{\y_n}(U_{P(n)})\right)\in \Omega_\go^n.\]
So by Jensen's inequality, and recalling \eqref{ineq:estimate of a_j D} and \eqref{ineq:improved D},
\begin{align*}
	|f(\x  \bs\xi)|^p
	&\le d^{p-1}\sum_{j=0}^{d-1}|a_{j}D^{m_j}|^p \left|\bE_{U, P}\left[\prod_{\ell=1}^{m_j}r(U_\ell)\cdot f\left(\bm\xi W_{Y_n'}^{(\x)}(U_P)\right)\right]\right|^p\\
	&\le d^{p-1}\sum_{j=0}^{d-1}|a_{j}D^{m_j}|^p \bE_{U, P}\left|f\left(\bm\xi W_{Y_n'}^{(\x)}(U_P)\right)\right|^p\\
	&\le d^{p-1}(C_1 \log(\go)+C_2)^{dp}\bE_{U, P}\left|f\left(\bm\xi W_{Y_n'}^{(\x)}(U_P)\right)\right|^p.
	\end{align*}
	Since $W_{Y_n'}^{(\x)}(U_P)$ takes values in $\Omega_\go^n$ and $\Omega_\go^n$ is a group, we have 
	\[
	\bE_{\bm\xi\sim \Omega_\go^n}\left|f\left(\bm\xi W_{Y_n'}^{(\x)}(U_P)\right)\right|^p=\|f\|_{L^p(\Omega_\go^n)}^p.
	\]
	Similarly, for any $\bm\xi\in \Omega_\go^n\subset \T^n$, we have 
	\[
\bE_{\x\sim\T^n}|f(\x  \bs\xi)|^p=\|f\|_{L^p(\T^n)}^p.\]
	All combined, we conclude 
	\begin{align*}
	\|f\|_{L^p(\T^n)}^p
	&=\bE_{\bm\xi\sim \Omega_\go^n}\bE_{\x\sim\T^n}|f(\x  \bs\xi)|^p\\
		&\le d^{p-1}(C_1 \log(\go)+C_2)^{dp}\bE_{U, P}\bE_{\x\sim\T^n}\bE_{\bm\xi\sim \Omega_\go^n}\left|f\left(\bm\xi W_{Y_n'}^{(\x)}(U_P)\right)\right|^p\\
		&= d^{p-1}(C_1 \log(\go)+C_2)^{dp}\|f\|_{L^p(\Omega_\go^n)}^p.
	\end{align*}
	This finishes the proof of the theorem. 
	\end{proof}

\printbibliography
 
\end{document}